\documentclass[11pt]{amsart}
\usepackage{amsfonts,amsmath,amssymb}
\usepackage{mathabx}
\usepackage{mathrsfs}

\usepackage[latin1]{inputenc}
\usepackage[usenames,dvipsnames]{pstricks}
\newtheorem{theorem}{Theorem}[section]

\newtheorem{Lemm}[theorem]{Lemma}

\newcommand{\p}{\partial}
\newcommand{\R}{\mathbb{R}}

\newcommand{\para}[1]{\left(#1\right)}

\usepackage{graphicx}

\usepackage{mathrsfs}
\usepackage[latin1]{inputenc}
\usepackage[T1]{fontenc}
\usepackage{color,xspace}
\usepackage{hyperref}
\usepackage[usenames,dvipsnames]{pstricks}
\usepackage{times}
\allowdisplaybreaks
\begin{document}
\title[Magnetic Schr\"odinger equation with time-dependent coefficient]
{Stability estimate for an inverse problem for the Schr\"{o}dinger equation in a magnetic field with time-dependent coefficient}
\author[I.~Ben A\"{\i}cha]{Ibtissem Ben A\"{\i}cha}
\address{I.~Ben A\"{\i}cha. University of Aix-Marseille, 58 boulevard Charles Livon, 13284 Marseille,  France. \& University of Carthage,
Faculty of Sciences of Bizerte, 7021 Jarzouna Bizerte, Tunisia. \& LAMSIN, National Engineering School of Tunis, B.P. 37, 1002 Tunis, Tunisia}
\email{ibtissem.benaicha@enit.utm.tn} \maketitle
\begin{abstract}
 We study the  stability issue in the inverse problem of determining the
 magnetic field and the time-dependent electric potential  appearing in the
 Schr\"{o}dinger equation, from boundary observations. We prove in dimension
$3$ or greater, that the knowledge of the Dicrichlet-to-Neumann map stably
determines the magnetic field and the electric potential.\\
 \textbf{keywords:} Stability estimates, Schr\"odinger equation,
magnetic field, time-dependent electric potential, Dirichlet-to-Neumann
map.
\end{abstract}

\section{Introduction }
\subsection{Statement of the problem}
The present paper deals with the inverse problem of
 determining  the magnetic field and
the time-dependent electric potential in the magnetic Schr\"odinger equation
from the knowledge of boundary observations. Let $\Omega \subset\R^{n}$, $n\geq 3$, be a
bounded and simply connected domain  with $\mathcal{C}^{\infty}$ boundary
$\Gamma$. We denote by $\Delta_{A}$ the Laplace operator associated to the
real valued magnetic potential $A\in \mathcal{C}^{3}(\Omega)$ which is defined 
by
$$\Delta_{A}=\sum_{j=1}^{n}(\p_{j}+ia_{j})^{2}=\Delta+2iA\cdot \nabla+i\,\mbox{div}(A)-|A|^{2}.$$
 Given $T>0$, we denote by
 $Q=\Omega\times (0,T)$
and $\Sigma=\Gamma\times (0,T)$. We consider the following initial boundary
problem for the Schr\"odinger equation
\begin{equation}\label{Eq1}
\left\{
  \begin{array}{ll}
   ( i\p_{t}+\Delta_{A}+q(x,t))u=0, & \mbox{in} \,Q, \\
    u(.,0)=u_{0}, & \mbox{in}\, \Omega,  \\
    u=f, & \mbox{on} \,\Sigma,
  \end{array}
\right.
\end{equation}
 where the real valued  bounded function $q\in W^{2,\infty}(0,T; W^{1,\infty}(\Omega))$ is the electric
potential. We  define the Dirichlet-to-Neumann map associated to the magnetic
Schr\"odinger equation (\ref{Eq1}) as
$$\begin{array}{ccc}
\Lambda_{A,q}: H^{2}(\Omega)\times H^{2,1}(\Sigma)&\longrightarrow& H^{1}(\Omega)\times L^{2}(\Sigma)\\
(u_{0},f)&\longmapsto&\displaystyle\Big(u(.,T),(\p_{\nu}+iA\cdot \nu)u\displaystyle\Big),
\end{array}$$
were $\nu(x)$ denotes the unit outward normal to $\Gamma$ at $x$, and
$\p_{\nu} u$ stands for $\nabla u\cdot\nu$. Here  $H^{2,1}(\Sigma)$ is a
Sobolev space we shall make precise below. We aim to know whether the knowledge of
the Dirichlet-to-Neumann map $\Lambda_{A,q}$  can uniquely determine the
magnetic and the electric potentials.

The problem of recovering coefficients in the magnetic Schr\"odinger equation
was treated by many authors. In \cite{[MC]}, Bellassoued and Choulli
considered the problem of recovering the magnetic potential $A$  from the
knowledge of the Dirichlet-to-Neumann map
 $\Lambda_{A}(f)=(\p_{\nu}+i\nu.A)u$  \,for \,$f\in L^{2}(\Sigma),$
associated to the Schr\"odinger equation with zero initial data. As it was
noted in \cite{[E]}, the Dirichlet-to-Neumann map $\Lambda_{A}$ is invariant
under the gauge transformation of the magnetic potential. Namely, given
$\varphi \in \mathcal{C}^{1}(\overline{\Omega})$ such that
$\varphi_{|\Gamma}=0$,  we have
 \begin{equation}\label{eq25}
 e^{-i\varphi}\Delta_{A} e^{i\varphi}=\Delta_{A+\nabla\varphi},\,\,\,\,\, e^{-i\varphi}\Lambda_{A}e^{i\varphi}=\Lambda_{A+\nabla\varphi},
\end{equation}
 and $\Lambda_{A}=\Lambda_{A+\nabla \varphi}$. Therefore, the magnetic potential $A$ can not be uniquely determined by the Dirichlet-to-Neumann map $\Lambda_{A}$. In geometric terms, the magnetic potential $A$ defines the connection given by the one form $\alpha_{A}=\sum_{j=1}^{n}a_{j}dx_{j}.$  The non uniqueness manifested in  (\ref{eq25}) says that the best one  can hope to recover from the Dirichlet-to-Neumann map is the 2-form 
 $$d\alpha_{A}=\sum_{i,j=1}^{n}\Big(\frac{\p a_{i}}{\p x_{j}}-\frac{\p a_{j}}{\p x_{i}} \Big)dx_{j}\wedge dx_{i},$$

 called the magnetic field. Bellassoued and Choulli proved in dimension $n\geq 2$ that the knowledge of the Dirichlet-to-Neumann map $\Lambda_{A}$ H\"older stably determines  the magnetic field $d\alpha_{A}$. 

In the presence of a time-independent electric potential, the inverse problem of determining the magnetic field $d\alpha_{A}$ and the electric potential
 $q$ from boundary observations was first considered by Sun \cite{[sun]}, in the case $n\geq 3$. He showed that $d\alpha_{A}$ and $q$ can be
uniquely determined when $A\in W^{2,\infty}$, $q\in L^{\infty}$ and $d\alpha_{A}$ is small in the $L^{\infty}$ norm. In \cite{[H]}, Benjoud studied the inverse  problem of recovering the magnetic field $d\alpha_{A}$ and the electric potential $q$ 
 from the knowledge of the Dirichlet-to-Neumann map. Assuming that the
 potentials are known in a neighborhood of the boundary, she proved a stability estimate with respect to arbitrary partial boundary observations.\\

In the Riemannian case,  Bellassoued \cite{[MB]} proved recently a 
H\"older-type stability estimate in the recovery of  the magnetic field
$d\alpha_{A}$ and the time-independent electric potential $q$ from the
knowledge of the Dirichlet-to-Neumann map associated to the Shr\"odinger
equation with zero initial data.  In the absence of the magnetic potential
$A$, the problem of recovering the electric potential $q$ on a compact
Riemannian manifold was solved by Bellassoued and Dos Santos Ferreira \cite{[BD2]}.\\

In recent years significant progress have been made in the recovery of time-dependent and  time-independent coefficients appearing in 
  hyperbolic equations, see for instance \cite{[BJY],[RS],[Is]}. We also refer to the work of Bellassoued and Benjoud \cite{[BH]} in which
 they prove that the Dirichlet-to-Neumann map determines
uniquely the magnetic field  in a magnetic wave equation. In \cite{[ES]},
Eskin proved that the Dirichlet-to-Neumann map uniquely determines
coefficients depending analytically on the time variable. In \cite{[Stef]},
Stefanov proved that the time-dependent potential $q$ appearing in the wave
equation is uniquely determined from the knowledge of scattering data. In
\cite{[Ram]}, Ramm and Sj\"ostrand proved a uniqueness result in recovering
the time-dependent potential $q$ from the Dirichlet-to-Neumann map, on the
infinite time-space cylindrical domain $\R_{t}\times\Omega$.   As for
stability results, we refer to Salazar \cite{[S]},  Waters \cite{[Waters]},  
Ben A\"icha \cite{[ibtissem]} and Kian \cite{[Yv]}.

  The problem of determining time-dependent electromagnetic potentials
appearing in a Schr\"odinger equation was treated by Eskin \cite{[E]}.
 Using a geometric optics construction, he prove the uniqueness for
this problem in domains with obstacles. In unbounded domains and in the
absence of the magnetic potential, Choulli , Kian and Soccorsi \cite{MYE}
treated the problem of recovering the time-dependent scalar potential $q$
appearing in the Schr\"odinger equation from boundary observations. Assuming that the domain is a 
$1$-periodic  cylindrical waveguide,  they proved logarithmic
stability for this problem.

 In the present paper, we  address the uniqueness and the stability issues in the inverse problem
 of recovering the magnetic field $d\alpha_{A}$ and the time-dependent potential $q$ in the dynamical Schr\"odinger equation, from the knowledge
of the operator $\Lambda_{A,q}$. By means of techniques used in
\cite{[MB],[H]}, we prove a "$\log$-type" stability estimate in the recovery of
the magnetic field and a "$\log$-$\log$-$\log$-type" stability inequality in the
determination of the
time-dependent electric potential.

 From a physical view point, our inverse problem consists in determining the magnetic field $d\alpha_{A}$
induced by the magnetic potential $A$, and the electric potential $q$ of an
inhomogeneous medium by probing it with disturbances generated on the
boundary. Here we assume that the medium is quiet initially and $f$ denotes the 
disturbance used to probe the medium. Our data are
 the response $(\p_{\nu}+iA.\nu)u$ performed on the boundary $\Sigma$, and the measurement $u(.,T)$, 
for different choices of $f$ and for all possible initial data $u_{0}$.
\subsection{Well-posedness of the magnetic Schr\"odinger equation and main results }
In order to state our main results, we need the following existence and
uniqueness result. To this end, we  introduce the following Sobolev
space
$$H^{2,1}(\Sigma)=H^{2}(0,T;L^{2}(\Gamma))\cap L^{2}(0,T;H^{1}(\Gamma)),$$
equipped with the norm
$$\|f\|_{H^{2,1}(\Sigma)}=\|f\|_{H^{2}(0,T;L^{2}(\Gamma))}+\|f\|_{L^{2}(0,T;H^{1}(\Gamma))},$$
and  we set
$$H^{2,1}_{0}(\Sigma)=\{f\in H^{2,1}(\Sigma), \,\,f(.,0)=\p_{t}f(.,0)=0\}.$$
Then we have the following theorem.
\begin{theorem}\label{Thm1.1}
Let $T>0$ and let $q\in W^{1,\infty}(Q)$, $A\in\mathcal{C}^{1}(\Omega)$ and
$u_{0}\in H^{1}_{0}(\Omega)\cap H^{2}(\Omega)$. Suppose that $f\in
H^{2,1}_{0}(\Sigma)$. Then, there exists a unique solution $u\in
\mathcal{C}(0,T; H^{1}(\Omega))$ of the Shr\"odinger equation (\ref{Eq1}).
Furthermore, we have $\p_{\nu}u\in L^{2}(\Sigma)$ and there exists a constant
$C>0$ such that
$$\|u(.,t)\|_{H^{1}(\Omega)}+\|\p_{\nu}u\|_{L^{2}(\Sigma)}\leq C\para{\|u_{0}\|_{H^{2}(\Omega)}+\|f\|_{H^{2,1}(\Sigma)}}.$$
\end{theorem}
As a corollary, the Dirichlet-to-Neumann map $\Lambda_{A,q}$ is 
bounded from $ H^{2}(\Omega)\times H^{2,1}(\Sigma)$ to $H^{1}(\Omega)\times
L^{2}(\Sigma).$
 The proof of Theorem\ref{Thm1.1} is given in Appendix A.\\
 In order to express the main results of
this article, we first define the following admissible sets of unknown
coefficients $A$ and $q$: for $\varepsilon >0$, $M>0$, we set
$$\mathcal{A}_{\varepsilon}=\{A\in C^{3}(\Omega),\,\,\,\|A\|_{W^{3,\infty}(\Omega)}\leq \varepsilon,\,\,\,\,\,\,\,A_{1}=A_{2}\,\,\mbox{in}\,\Gamma\},$$
$$\mathcal{Q}_{M}=\{q\in \mathcal{X}=W^{2,\infty}(0,T;W^{1,\infty}(\Omega)),\,\,\,\|q\|_{\mathcal{X}}\leq M,\,\,\,\,\,\,q_{1}=q_{2}\,\,\,\,\mbox{in}\,\Gamma\}.$$
Our first main result claims stable determination of the magnetic field
$d\alpha_{A}$, from full boundary measurement $\Lambda_{A,q}$ on the
cylindrical domain $Q$.
\begin{theorem}\label{Thm1} Let $\alpha>\frac{n}{2}+1$.  Let $q_{i}\in
\mathcal{Q}_{M}$, $A_{i}\in \mathcal{A}_{\varepsilon}$, such that
$\|A_{i}\|_{H^{\alpha}(\Omega)}\leq M$, for $i=1,\,2$. Then, there exist three
constants $C>0$ and $\mu,s\in(0,1),$ such that we have
$$\|d{\alpha_{A_{1}}}-d{\alpha_{A_{2}}}\|_{L^{\infty}(\Omega)}\leq C\para{\|\Lambda_{A_{2},q_{2}}-\Lambda_{A_{1},q_{1}}\|^{1/2}+|\log \|\Lambda_{A_{2},q_{2}}-\Lambda_{A_{1},q_{1}}\||^{-\mu}}^{s}.$$
Here $C$  depends only on $\Omega$, $\varepsilon$, $M$ and $T$.
\end{theorem}
Next, assuming that the magnetic potential $A$ is divergence free,  we
can stably retrieve the electric potential.
\begin{theorem}\label{Thm2}
Let $q_{i}\in \mathcal{Q}_{M}$, $A_{i}\in \mathcal{A}_{\varepsilon}$,   for
$i=1,\,2$. Assume that div $A_{i}=0$. Then there exist three constants $C>0$, and 
$m, \mu\in(0,1)$, such that we have
$$\|q_{1}-q_{2}\|_{H^{-1}(Q)}\leq
C \Phi_{m}(\|\Lambda_{A_{2},q_{2}}-\Lambda_{A_{1},q_{1}}\|),$$

where
$$
\Phi_{m}(\|\Lambda_{A_{2},q_{2}}-\Lambda_{A_{1},q_{1}}\|)=
\left\{
\begin{array}{lll}
 | \log\,|\log|\log\|\Lambda_{A_{2},q_{2}}-\Lambda_{A_{1},q_{1}}\||^{\mu}  |\,|^{-1} &\,\mbox{if}\,\,\|\Lambda_{A_{2},q_{2}}-\Lambda_{A_{1},q_{1}}\|<m,\\
 \\
\displaystyle\frac{1}{m} \|\Lambda_{A_{2},q_{2}}-\Lambda_{A_{1},q_{1}}\|   &\,\mbox{if}\,\,\|\Lambda_{A_{2},q_{2}}-\Lambda_{A_{1},q_{1}}\|\geq m.
\end{array}
\right.
$$
Here $C$  depends on  $\Omega$, $M$, $\varepsilon$ and $T$.

\end{theorem}

 The text is organized as follows. Section \ref{Sec2} is devoted to the construction of special
geometrical optics solutions to the Shr\"odinger equation (\ref{Eq1}). Using
these particular solutions, we establish in sections \ref{Sec3} and \ref{Sec4}, two
stability estimates for the magnetic field and the electric potential. In
Appendix A, we develop the proof of Theorem\ref{Thm1.1}. Appendix B contains the proof of several 
 technical results used in the derivation of the main results.
\section{Preliminaries and geometrical optics solutions }\label{Sec2}
The present section is devoted to the construction of suitable geometrical
optics solutions, which are key ingredients in the proof of our main results.
We start  by collecting several known lemmas from \cite{[R1],[R2]}.
\subsection{Preliminaries}
Let $\omega=\omega_{\Re}+i\omega_{\Im}$ be a vector with
$\omega_{\Re},\,\omega_{\Im}\in \mathbb{S}^{n-1}$, and
$\omega_{\Re}\cdot\omega_{\Im}=0$. We shall see that the differential
operator $N_{\omega}=\omega\cdot\nabla$ is invertible
and we have
$$N_{\omega}^{-1}(g)(x)=\displaystyle\frac{1}{({2\pi})^{n}}\displaystyle\int_{\R^{n}}e^{-ix\cdot\xi}\displaystyle
\para{\displaystyle\frac{\hat{g}(\xi)}{\omega\cdot\xi}}d\xi=\frac{1}{2\pi}\int_{\R^{2}}\frac{1}{y_{1}+iy_{2}}g(x-y_{1}\omega_{\Re}-y_{2}\omega_{\Im})
\,dy_{1}\,dy_{2}.$$ Notice that the differential operator 
$\overline{\p}$ corresponds to $N_{\omega}$ with $\omega=(0,1)$.
\begin{Lemm}\label{Lm2.1}
Let $r>0$, $k>0$ and let $g \in W^{k,\infty}(\R^{n})$ be such that Supp $g
\subseteq B(0,r)=\{x\in\R^{n},\,\,\,|x|\leq r\}$. Then the function
 $\phi = N_{\omega}^{-1}(g) \in W^{k,\infty}(\R^{n})$ solves $N_{\omega}(\phi)=g$, and satisfies the estimate
$$\|\phi\|_{W^{k,\infty}(\R^{n})}\leq C \,\|g\|_{W^{k,\infty}(\R^{n})},$$
where $C$ is a positive constant depending only on $r$.
\end{Lemm}
We recall  from \cite{[Sa]}, the following technical result.
\begin{Lemm}\label{Lm2.2}
Let $A\in C_{c}(\R^{n})$, $\xi\in\R^{n}$, and
$\omega=\omega_{\Re}+i\omega_{\Im}$ with
$\omega_{\Re},\,\omega_{\Im}\in\mathbb{S}^{n-1}$ and
$\omega_{\Re}\cdot\omega_{\Im}=\omega_{\Re}\cdot\xi=\omega_{\Im}\cdot\xi=0$.
Then we have the following identity
$$\int_{\R^{n}}\omega\cdot A(x)e^{iN_{\omega}^{-1}(-\omega\cdot A)(x)}e^{i\xi\cdot x}\,dx=\int_{\R^{n}}\omega\cdot A(x)e^{i\xi\cdot x}dx.$$
\end{Lemm}
\subsection{Geometrical optics solutions}
In this section, we build special solutions to the magnetic Schr\"odinger
equation (\ref{Eq1}), inspired by techniques used in elliptic problems. For
this purpose, we consider a vector $\omega=\omega_{\Re}+i\,\omega_{\Im}$,
such that $\omega_{\Re},\,\omega_{\Im}\in \mathbb{S}^{n-1}$ and
$\omega_{\Re}\,.\,\omega_{\Im}=0$. For $\sigma>1$, we define the complex
variable $\rho$ as follows
\begin{equation}\label{Eq2}
\rho=\sigma\omega+y,
\end{equation}
where $y\in B(0,1)$ is fixed and  independent of $\sigma$. In what follows,
$P(D)$ denotes a differential operator with constant coefficients:
$$P(D)=\sum_{|\alpha|\leq m}a_{\alpha}\,D^{\alpha},\,\,\,\,\,\,\,\,\,\,D=-i(\p_{t},\p_{x}).$$
  We associate to  the operator $P(D)$ its symbol $p(\xi,\tau)$  defined by
$$p(\xi,\tau)=\sum_{|\alpha|\leq m}
a_{\alpha}(\xi,\tau)^{\alpha},\,\,\,\,\,\,\,\,\,\,(\xi,\tau)\in\R^{n+1}.$$
Moreover,  we set
 $$\widetilde {p}(\xi,\tau)=\para{\sum_{\beta\in\mathbb{N}} \sum_{\alpha\in\mathbb{N}^{n}}|\p_{\tau}^{\beta}\p_{\xi}^{\alpha}p(\xi,\tau)|^{2}}^{\frac{1}{2}},\,\,
\,\,\,(\xi,\tau)\in\R^{n+1},$$  and introduce the operators
$$\Delta_{\rho}=\Delta-2i\rho\cdot\nabla\,\,\,\,\,\, \mbox{and}\,\,\,\,\,\, \nabla_{\rho}=\nabla-i\rho.$$ 
We turn now to building particular solutions to the magnetic Shr\"odinger equation. We
proceed with a succession of lemmas. The first result is inspired by
H\"ormander \cite{[L.H]} (see Appendix B).
\begin{Lemm}\label{Lemme 2.3}
Let $P\neq0$ be an operator. There exists a linear operator $E\in\mathcal{L}(
L^{2}(0,T;H^{1}(\Omega))),$
 such that:

   $$P(D)E f=f, \,\,\,\,\,\,\mbox{for\,\,any}\,\, f\in L^{2}(0,T;H^{1}(\Omega)).$$
Moreover, for any linear operator $S$ with constant coefficients such that
      $\displaystyle\frac{|S(\xi,\tau)|}{\tilde{p}(\xi,\tau)}$ is bounded
      in $\R^{n+1}$, we have the following estimate
\begin{equation}\label{Eq2.5}
\|S(D)E f\|_{L^{2}(0,T;
      H^{1}(\Omega))}\leq C \,\displaystyle\sup_{\R^{n+1}}\frac{|S(\xi,\tau)|}{\tilde{p}(\xi,\tau)}\|f\|_{L^{2}(0,T;H^{1}(\Omega))}.
\end{equation}
Here $C$  depends only on the degree of $P$, $\Omega$ and $T$.
\end{Lemm}
\begin{Lemm}\label{Lemme 2.4}
  There exists a bounded operator $E_{\rho}:L^{2}(0,T;H^{1}(\Omega))\longrightarrow
L^{2}(0,T;H^{2}(\Omega))$ such that
$$P_{\rho}(D) E_{\rho}f=(i\p_{t}+\Delta_{\rho})E_{\rho}f=f\quad\mbox{for\,any}\quad f\in L^{2}(0,T;H^{1}(\Omega)).$$
Moreover, there exists a constant $C(\Omega,T)>0$ such that
\begin{equation}\label{Eq2.6}
\|E_{\rho}f\|_{L^{2}(0,T;H^{k}(\Omega))}\leq \frac{C}{\sigma^{2-k}}\|f\|_{L^{2}(0,T;H^{1}(\Omega))},\,\,\,\,\,k=1,\,2.
\end{equation}
 \end{Lemm}
\begin{proof}{}
From Lemma \ref{Lemme 2.3}, we deduce the  existence of  a linear operator
$E_{\rho}\in \mathcal{L}\Big(L^{2}(0,T;H^{1}(\Omega))\Big)$ such that
$P_{\rho}(D) E_{\rho}f=f$.
 Moreover, since $|\widetilde{p _{\rho}}(\xi,\tau)|>\sigma$,
we get from (\ref{Eq2.5})
 \begin{equation}\label{aj1}
\|E_{\rho}f\|_{L^{2}(0,T;H^{1}(\Omega))}\leq \frac{C}{
\sigma} \|f\|_{L^{2}(0,T;H^{1}(\Omega))}.
\end{equation}
 Similarly, since $\displaystyle\frac{|\xi|}{\widetilde{p_{\rho}}(\xi,\tau)}$ is bounded on
$\R^{n+1}$, we get
$$\|\nabla E_{\rho}f\|_{L^{2}(0,T;H^{1}(\Omega))}\leq C
\|f\|_{L^{2}(0,T;H^{1}(\Omega))}.$$ From this and (\ref{aj1}) we see that $E_{\rho}$ is bounded from $L^{2}(0,T;H^{1}(\Omega))$ into $L^{2}(0,T;H^{2}(\Omega))$.
\end{proof}
Let us now deduce the coming statementfrom the above lemma.
 \begin{Lemm}\label{Lemme 2.5}
There exists $\varepsilon>0$ such that for all $A\in W^{1,\infty}(\Omega)$ obeying 
$\|A\|_{W^{1,\infty}(\Omega)}\leq \varepsilon,$  we may build a bounded operator
$F_{\rho}:L^{2}(0,T;H^{1}(\Omega))\longrightarrow L^{2}(0,T;H^{2}(\Omega))$
such that:
\begin{equation}\label{aj3}
\big(i\p_{t}+\Delta_{\rho}+2iA\cdot\nabla\big)
F_{\rho}f=f,\quad\mbox{for\,any}\quad f\in L^{2}(0,T;H^{1}(\Omega)).
\end{equation}
Moreover, there exists a constant $C(\Omega, T)>0$ such  that
\begin{equation}\label{eq2.9}
\|F_{\rho}f\|_{L^{2}(0,T;H^{k}(\Omega))}\leq \frac{C}{\sigma^{2-k}}\|f\|_{L^{2}(0,T;H^{1}(\Omega))},\,\,\,\,\,k=1,\,2.
\end{equation}
\end{Lemm}
 \begin{proof}{}
 Let $f\in L^{2}(0,T; H^{1}(\Omega)).$
We start by introducing the following operator
 $$\begin{array}{rrr}
 S_{\rho}: L^{2}(0,T; H^{2}(\Omega))&\longrightarrow &L^{2}(0,T; H^{2}(\Omega))\\
 g&\longmapsto& E_{\rho}(-2i A\cdot\nabla g+f).
 \end{array}$$
 Since $\|A\|_{W^{1,\infty}(\Omega)}\leq
\varepsilon,$ we deduce from (\ref{Eq2.6}) with $k=2$ that
 \begin{eqnarray}\label{Equation 2.9}
 \|S_{\rho}(h)-S_{\rho}(g)\|_{L^{2}(0,T; H^{2}(\Omega))}&\leq& C\varepsilon \|h-g\|_{L^{2}(0,T; H^{2}(\Omega))},
 \end{eqnarray}
 for any $h,\,g\in L^{2}(0,T;H^{2}(\Omega))$. Thus, $S_{\rho}$ is a contraction from $L^{2}(0,T;H^{2}(\Omega))$ into $L^{2}(0,T;H^{2}(\Omega))$ for $\varepsilon$ small enough. Then, $S_{\rho}$ admits  a unique fixed
point $g\in L^{2}(0,T; H^{2}(\Omega))$. Put $F_{\rho}f=g$. It is clear that
$F_{\rho}f$ is a solution to (\ref{aj3}). Then, taking into account the
identity $S_{\rho}F_{\rho} f=E_{\rho}(-2iA\cdot \nabla F_{\rho}f+f)$ and the
estimate (\ref{Equation 2.9}), we get
 $$\begin{array}{lll}
 \|F_{\rho}f\|_{L^{2}(0,T;H^{2}(\Omega))}&=&\|S_{\rho}F_{\rho}f-S_{\rho}(0)\|_{L^{2}(0,T; H^{2}(\Omega))}+\|S_{\rho}(0)\|_{L^{2}(0,T; H^{2}(\Omega))} \\
&\leq& C \varepsilon
\|F_{\rho}f\|_{L^{2}(0,T;H^{2}(\Omega))}+\|E_{\rho}f\|_{L^{2}(0,T;H^{2}(\Omega))}.
\end{array}$$
From this and (\ref{Eq2.6}) with $k=2$, we end up getting for $\varepsilon$ small enough 
\begin{equation}\label{l'aequation 2.11}
\|F_{\rho}f\|_{L^{2}(0,T;H^{2}(\Omega))}\leq
{C}\|f\|_{L^{2}(0,T;H^{1}(\Omega))}.
\end{equation}
This being said, it remains to show (\ref{eq2.9}) for $k=1$. To see this, we
notice from (\ref{Eq2.6}) with $k=1$ that
$$\begin{array}{lll}
\|F_{\rho}f\|_{L^{2}(0,T;H^{1}(\Omega))}&\leq& \|E_{\rho}(-2iA\cdot \nabla F_{\rho}f+f)\|_{L^{2}(0,T;H^{1}(\Omega))}\\
&\leq& \displaystyle\frac{C}{\sigma}\para{\varepsilon\|F_{\rho}f\|_{L^{2}(0,T;H^{2}(\Omega))}+\|f\|_{L^{2}(0,T;H^{1}(\Omega))}}.
\end{array}$$

Then the estimate (\ref{eq2.9}) for $k=1$ follows readily from  this and
(\ref{l'aequation 2.11}).

 \end{proof}
\begin{Lemm}\label{Lemme 2.6}
There exists $\varepsilon>0$ such that for all $A\in W^{1,\infty}(\Omega)$ obeying 
$\|A\|_{W^{1,\infty}(\Omega)}\leq \varepsilon,$  we may build a bounded operator
$G_{\rho}:L^{2}(0,T;H^{1}(\Omega))\longrightarrow L^{2}(0,T;H^{2}(\Omega))$
such that:
\begin{equation}\label{l'equation 2.12}
\big(i\p_{t}+\Delta_{\rho}+2i A\cdot\nabla_{\rho}\big)G_{\rho}f=f\quad \mbox{for\,any}\quad f\in L^{2}(0,T;H^{1}(\Omega)).
\end{equation}
Moreover, there exists a constant $C(\Omega,T)>0 $ such that
\begin{equation}\label{l'equation 2.13}
\|G_{\rho}f\|_{L^{2}(0,T;H^{k}(\Omega))}\leq \frac{C}{\sigma^{2-k}}\|f\|_{L^{2}(0,T;H^{1}(\Omega))},
\,\,\,\,\,\,k=1,\,2.
\end{equation}
\end{Lemm}
\begin{proof}{}
Let $f\in L^{2}(0,T;H^{1}(\Omega))$. We introduce the following operator
$$\begin{array}{rrr}
R_{\rho}:L^{2}(0,T;H^{1}(\Omega))&\longrightarrow& L^{2}(0,T;H^{1}(\Omega))\\
g&\longmapsto& F_{\rho}(-2\rho\cdot A g+f)
\end{array}$$
From (\ref{Eq2}), we see that $|\rho|< 3\sigma$. Thus, arguing as in the
proof of Lemma \ref{Lemme 2.5}, we prove the existence of a unique solution
$G_{\rho}f=g$ to the equation (\ref{l'equation 2.12}). Moreover there exists
a positive constants $C>0$ such that we have
\begin{equation}\label{l'equation 2.14}
 \|u\|_{L^{2}(0,T,H^{1}(\Omega))}\leq \frac{C}{\sigma}\|f\|_{L^{2}(0,T;H^{1}(\Omega))}.
\end{equation}
Further, combining the definition of $R_{\rho}$ with (\ref{aj3}) we deduce
(\ref{l'equation 2.13}) for $k=2$.
\end{proof}
Armed with lemma \ref{Lemme 2.6}, we are now in position to establish the
main result of this section, which can
be stated as follows
\begin{Lemm}\label{Prop3.1} Let $M>0$, $\varepsilon>0$, $\omega\in\mathbb{S}^{n-1}$ and   $A\in \mathcal{A}_{\varepsilon}$ satisfy $\|A\|_{W^{1,\infty}(\Omega)}\leq
\varepsilon$. Put $\phi=N_{\omega}^{-1}(-\omega.A)$. Then, for all
$\sigma\geq \sigma_{0}>0$ the magnetic Schr\"odinger equation
\begin{equation}\label{Eq6}
(i\p_{t}+\Delta_{A}+q(x,t))u(x,t)=0,\,\,\,\,\,\mbox{in}\,\,Q
\end{equation}
admits a solution $u\in H^{2}(0,T;H^{1}(\Omega))\cap
L^{2}(0,T;H^{2}(\Omega)),$ of the form
\begin{equation}\label{Eq7}
u(x,t)=e^{-i\big((\rho\cdot\rho)t+x\cdot\rho\big)}\big(e^{i\phi(x)}+w(x,t)\big),
\end{equation}
 in such a way that
\begin{equation}\label{eq}
\omega\cdot\nabla\phi(x)=-\omega\cdot A(x),\,\,\,\,\,x\in\R^{n}.
\end{equation}
Moreover, $w\in H^{2}(0,T;H^{1}(\Omega))\cap L^{2}(0,T;H^{2}(\Omega))$
satisfies
\begin{equation}\label{Equation 2.17}
\sigma\|w\|_{H^{2}(0,T;H^{1}(\Omega))}+\|w\|_{L^{2}(0,T;H^{2}(\Omega))}\leq C,
\end{equation}
where the constants $C$ and $\sigma_{0}$  depend only on $\Omega, T$ and
$M.$
\end{Lemm}
Here we extended $A$ by zero outside $\Omega$.
\begin{proof}{} To prove our lemma, it is enough to show that $w\in H^{2}(0,T;H^{1}(\Omega))\cap L^{2}(0,T;H^{2}(\Omega)) $
satisfies the estimate (\ref{Equation 2.17}). Substituting (\ref{Eq7}) into
the equation (\ref{Eq6}), one gets
$$\begin{array}{lll}
\displaystyle\Big(i\p_{t}+\Delta_{\rho}+2iA(x)\cdot\nabla_{\rho}+h(x,t)\Big)w(x,t)
\!\!\!&=&\!\!\!-e^{i\phi(x)}\Big(i\Delta\phi(x)-|\nabla\phi(x)|^{2}+2\sigma\omega\cdot \nabla\phi(x)+2\sigma\omega\cdot A(x)\\
&&\,\,\,\,\,+2y\cdot \nabla\phi(x)+2A(x)\cdot y-2A(x)\cdot \nabla\phi(x)+h(x,t)\Big),
\end{array}$$
where $h(x,t)=i\mbox{div}A(x)-|A(x)|^{2}+q(x,t)$.  Equating coefficients of
power of $|\sigma|$ to zero, we get $\omega\cdot\nabla\phi(x)=-\omega\cdot
A(x)$ for all $x\in\R^{n}.$ Then $w$ solves the following equation
\begin{equation}\label{Equation 2.19}
\para{i\p_{t}+\Delta_{\rho}+2iA(x)\cdot\nabla_{\rho}+h(x,t)}w(x,t)=L(x,t),
\end{equation}
where
\begin{equation}\label{h}
L(x,t)=-e^{i\phi(x)}\big(i\Delta\phi(x)-|\nabla\phi(x)|^{2}+2y\cdot\nabla\phi(x)+2A(x)\cdot
y-2A\cdot\nabla\phi(x)+h(x,t)\big).
\end{equation}
In light of (\ref{Equation 2.19}), we introduce the following map
$$\begin{array}{rrr}
U_{\rho}:L^{2}(0,T;H^{1}(\Omega))&\longrightarrow& L^{2}(0,T;H^{1}(\Omega)),\\
w&\longmapsto&G_{\rho}(-w\,h+L).
\end{array}$$
Applying (\ref{l'equation 2.13}) with $k=1$ and $f=h\,(w-\tilde{w})$, we get
for all $w,\tilde{w}\in L^{2}(0,T;H^{1}(\Omega))$ that
$$\begin{array}{lll}
\|U_{\rho}(w)-U_{\rho}(\tilde{w})\|_{L^{2}(0,T;H^{1}(\Omega))}&=&\|G_{\rho}(h\,(w-\tilde{w}))\|_{L^{2}(0,T;H^{1}(\Omega))}\\
&\leq&\displaystyle\frac{C}{\sigma}\|h\|_{\mathcal{X}}\|w-\tilde{w}\|_{L^{2}(0,T;H^{1}(\Omega))}.
\end{array}$$
Taking $\sigma_{0}$ sufficiently large so that
$\sigma_{0}>2C\|h\|_{\mathcal{X}},$ then, for each $\sigma>\sigma_{0}$, $U_{\rho}$
admits a unique fixed point $w\in L^{2}(0,T;H^{1}(\Omega))$ such that
$U_{\rho}(w)=w$. Again, applying (\ref{l'equation 2.13}) with $k=1$ and
$f=-hw+L$, one gets
$$\begin{array}{lll}
\|w\|_{L^{2}(0,T;H^{1}(\Omega))}&=&\|G_{\rho}(-hw+L)\|_{L^{2}(0,T;H^{1}(\Omega))}\\
&\leq&\displaystyle \frac{1}{2}\|w\|_{L^{2}(0,T;H^{1}(\Omega)}+\frac{C}{\sigma}\|L\|_{L^{2}(0,T;H^{1}(\Omega))}.
\end{array}$$
Therefore, in view of Lemma \ref{Lm2.1} and (\ref{h}), we get
\begin{equation}\label{w}
\|w\|_{L^{2}(0,T;H^{1}(\Omega))}
\leq \displaystyle\frac{C}{\sigma}.
\end{equation}
Next, differentiating  the equation (\ref{Equation 2.19}) twice with respect
to $t$,
 taking
into account that $\|h\|_{\mathcal{X}}$ is uniformly bounded with respect to
$\sigma$, and proceeding as before, we show that
\begin{equation}\label{wt}
\|\p_{t}^{k}w\|_{L^{2}(0,T;H^{1}(\Omega))}\leq
\frac{C}{\sigma},\,\,\,\,\,k=1,2.\end{equation}
Finally, from (\ref{w}) and  Lemma \ref{Lm2.1}, we obtain
\begin{eqnarray}\label{l'equation 2.23}
\|w\|_{L^{2}(0,T;H^{2}(\Omega))}&\leq& C\|-wh+L\|_{L^{2}(0,T;H^{1}(\Omega))}\cr
&\leq&C\Big(\displaystyle\frac{C}{\sigma}\|h\|_{\mathcal{X}}+C \Big)\cr
&\leq&C,
\end{eqnarray}
 by applying (\ref{l'equation
2.13}) with $k=2$ and $f=-wh+L$. Thus, we get the desired result by combining (\ref{w})-(\ref{l'equation
2.23}).

\end{proof}
\section{Stability estimate for the magnetic field}\label{Sec3}
In this section, we prove Theorem\ref{Thm1} by means of the  geometrical
optics solutions
\begin{equation}\label{Eq10}
u_{j}(x,t)=e^{-i\big((\rho_{j}\cdot\rho_{j})t+x\cdot\rho_{j}\big)}\Big(e^{i\phi_{j}(x)}+w_{j}(x,t)\Big),\quad
j=1,2,
\end{equation}
 associated $A_{j}$ and $q_{j}$. Here we choose
$\rho_{j}=\sigma\omega_{j}$ and we recall that the correction term $w_{j}$
satisfies (\ref{Equation 2.17}) and
 that $\phi_{j}(x)=N^{-1}_{\omega_{j}^{*}}(-\omega_{j}^{*}.A_{j})$ solves the
transport equation
$$\omega_{j}^{*}.\nabla \phi_{j}(x)=-\omega_{j}^{*}.A(x),\,\,\,\,\,\,x\in\R^{n}.$$
Let us  specify the choice of $\rho_{j}$: we consider $\xi\in\R^{n}$ and $\omega=\omega_{\Re}+i\omega_{\Im}$ with $\omega_{\Re},\,\omega_{\Im}\in \mathbb{S}^{n-1}$ and $\omega_{\Re}.\omega_{\Im}=\xi.\omega_{\Re}=\xi.\omega_{\Im}=0$. for each $\sigma>|\xi|/{2}$, we
denote
\begin{equation}\label{Eq8}
\rho_{1}=\sigma\para{i\omega_{\Im}+\para{-\frac{\xi}{2\sigma}+\sqrt{1-\frac{|\xi|^{2}}{4\sigma^{2}}}\omega_{\Re}}}=\sigma\omega_{1}^{*},
\end{equation}
\begin{equation}\label{Eq9}
\rho_{2}=\sigma\para{-i\omega_{\Im}+\para{\frac{\xi}{2\sigma}+\sqrt{1-\frac{|\xi|^{2}}{4\sigma^{2}}}\omega_{\Re}}}=\sigma\omega_{2}^{*}.
\end{equation}
Notice that $\rho_{j}.\rho_{j}=0.$ In this  section, we aim for recovering
the magnetic field $d\alpha_{A}$ from the boundary operator
$$\begin{array}{ccc}
\Lambda_{A,q}:L^{2}(\Omega)\times H^{2,1}(\Sigma)&\longrightarrow& H^{1}(\Omega)\times L^{2}(\Sigma)\\
g=(u_{0},f)&\longmapsto&\displaystyle\Big(u(.,T),(\p_{\nu}+iA\cdot\nu)u\displaystyle\Big).
\end{array}$$
We denote by
$$\Lambda_{A,q}^{1}=u(.,T),\quad \Lambda_{A,q}^{2}=(\p_{\nu}+iA\cdot \nu)u.$$
 We first establish an orthogonality identity for the magnetic potential $A=A_{1}-A_{2}$.
\subsection{A basic identity for the magnetic potential}
In this section, we derive an identity relating the magnetic
potential $A$ to the solutions $u_{j}$. We start by the following  result.
\begin{Lemm}\label{Lm4.2}
Let $\varepsilon>0$, $A_{j}\in\mathcal{A}_{\varepsilon}$ and  $u_{j}$ be the
solutions given by (\ref{Eq10}) $j=1,\,2$. Then, for all $\xi\in\R^{n}$ and
$\sigma>max(\sigma_{0},|\xi|/2)$, we have
$$\int_{Q}iA(x)\cdot\big(\overline{u_{1}}\nabla u_{2}-u_{2}\nabla \overline{u_{1}}\big)\,dx\,dt
=\int_{Q}A(x)\cdot(\rho_{2}+\overline{\rho_{1}})e^{-ix\cdot\xi}e^{i(\phi_{2}-\overline{\phi_{1}})(x)}+I(\xi,\sigma),$$
where the remaining term $I(\xi,\sigma)$ is uniformly bounded with respect to
$\sigma$ and $\xi$.
\end{Lemm}
\begin{proof}{}
In light of  (\ref{Eq10}), we have by direct computation
$$\begin{array}{lll}
\overline{u_{1}}\nabla u_{2}-u_{2}\nabla\overline{u_{1}}&=&e^{-ix\cdot(\rho_{2}-\overline{\rho_{1}})}\Big[-i \rho_{2}e^{i(\phi_{2}-\overline{\phi_{1}})}
-i\overline{\rho_{1}}e^{i(\phi_{2}-\overline{\phi_{1}})}\\
&&+i\nabla \phi_{2} e^{i(\phi_{2}-\overline{\phi_{1}})}+i\nabla \overline{\phi_{1}}e^{i(\phi_{2}-\overline{\phi_{1}})}-i\rho_{2}w_{2}e^{-i
 \overline{\phi_{1}}}-i\overline{\rho_{1}}\overline{w_{1}}e^{i\phi_{2}}\\
&&+\nabla w_{2}e^{-i\overline{\phi_{1}}}-\nabla \overline{w_{1}}e^{i\phi_{2}}-i\rho_{2}\overline{w_{1}}e^{i\phi_{2}}-i\overline{\rho_{1}}
w_{2}e^{-i\overline{\phi_{1}}}+i\nabla \phi_{2} \overline{w_{1}}e^{i\phi_{2}}\\
&&+iw_{2}\nabla \overline{\phi_{1}}e^{-i\overline{\phi_{1}}}-i\rho_{2}w_{2}\overline{w_{1}}-i\overline{\rho_{1}}\overline{w_{1}}w_{2}
+\nabla w_{2}\overline{w_{1}}-\nabla \overline{w_{1}} w_{2}\Big].\\
\end{array}$$
Therefore, as we have $\rho_{2}-\overline{\rho_{1}}=\xi$, this yields that
$$\begin{array}{lll}
\displaystyle\int_{Q}iA(x)\cdot\para{\overline{u_{1}}\nabla u_{2}-u_{2}\nabla\overline{u_{1}}}dx\,dt&=&
\displaystyle\int_{Q}
A(x)\cdot(\rho_{2}+\overline{\rho_{1}})e^{-ix.\xi}e^{i(\phi_{2}-\overline{\phi_{1}})}\,dx\,dt+I(\xi,\sigma),
\end{array}$$
where
$I(\xi,\sigma)=\displaystyle\int_{Q}iA(x)\cdot\Big(\psi_{1}(x,t)+\psi_{2}(x,t)\Big)\,dx\,dt,$
and  $\psi_{1}$, $\psi_{2}$ stand for
$$\psi_{1}=-i(\rho_{2}+\overline{\rho_{1}})\left( w_{2}e^{-i\overline{\phi_{1}}}+\overline{w_{1}}e^{i\phi_{2}}+w_{2}\overline{w_{1}}\right),$$
$$\begin{array}{lll}
\psi_{2}&=&e^{i\phi_{2}}\big(i\nabla\phi_{2}\overline{w_{1}}-\nabla \overline{w_{1}}\big)+e^{-i\overline{\phi_{1}}}\big(\nabla w_{2}
+i\nabla\overline{\phi_{1}}w_{2}\big)
+\nabla w_{2}\overline{w_{1}}-\nabla \overline{w_{1}}w_{2}
+i\big(\nabla\phi_{2}+\nabla\overline{\phi_{1}}\big)e^{i(\phi_{2}-\overline{\phi_{1}})}.
\end{array}$$
In view of bounding  $|I(\xi,\sigma)|$ uniformly with respect to $\xi$ and
$\sigma$, we use the fact that  $A$ is extended by zero outside $\Omega$ and use Lemma \ref{Lm2.1}
to get
$$\|\phi_{j}\|_{L^{\infty}(\Omega)}\leq C \|A_{j}\|_{L^{\infty}(\R^{n})}\leq C \varepsilon,\,\,\,\,\,\,\,j=1,\,2.$$
Recalling  (\ref{eq}) and (\ref{Equation 2.17}) and applying Lemma \ref{Lm2.1}, we
get
\begin{equation}\label{mag1}\|\psi_{j}\|_{L^{1}(Q)}\leq C\para{C +\frac{1}{\sigma}}\leq
C,\,\,\,\,\,\,\,j=1,\,2,
\end{equation}
which yields the desired result.
\end{proof}
With the help of the above lemma we may now derive the following
orthogonality  identity for the magnetic potential.
\begin{Lemm}\label{Lm4.3}
Let $\xi\in\R^{n}$ and  $\sigma>max(\sigma_{0},|\xi|/{2})$. Then, we have the following
identity
$$\begin{array}{lll}
\displaystyle\int_{Q}A(x)\!\!\!\!\!&\cdot&\!\!\!\! (\rho_{2}+\overline{\rho_{1}})e^{-ix\cdot\xi}e^{i(\phi_{2}-\overline{\phi_{1}})}\,dx\,dt
=2\sigma T\displaystyle\int_{\Omega}\overline{\omega}\cdot A(x) e^{-ix\cdot\xi}\,dx+J(\xi,\sigma),
\end{array}$$
with $|J(\xi,\sigma)|\leq C|\xi|$, where $C$ is independent of $\sigma$ and
$\xi$.
\end{Lemm}
\begin{proof}{}
In view of (\ref{Eq8}) and (\ref{Eq9}), we have
\begin{eqnarray}\label{mag3}
\displaystyle\int_{Q}A(x)\!\!\!\!\!&\cdot&\!\!\!\!\!(\rho_{2}+\overline{\rho_{1}})
e^{-ix\cdot\xi}e^{i (\phi_{2}-\overline{\phi_{1}})} \,dx\,dt
=2\sigma \displaystyle\int_{Q}\overline{\omega}\cdot A(x) e^{-ix\cdot\xi}e^{i(\phi_{2}-\overline{\phi_{1}})}\,dx\,dt\cr
&&-2\sigma \displaystyle\para{1-\displaystyle\sqrt{1-\displaystyle|\xi|^{2}/4\sigma^{2}}}
\displaystyle\int_{Q}\omega_{\Re}\cdot A(x)e^{-ix\cdot\xi}e^{i(\phi_{2}-\overline{\phi_{1}})}\,dx\,dt,
\end{eqnarray}
where we recall that
$$\overline{\phi_{1}}=N^{-1}_{\overline{\omega_{1}^{*}}}(-\overline{\omega_{1}^{*}}\cdot A_{1}),\,\,\,\,\,\,\,\,\,\,\,\phi_{2}=N^{-1}_{\omega^{*}_{2}}(-\omega^{*}_{2}\cdot
A_{2}).$$  Set
$\overline{\Psi_{1}}=N^{-1}_{\overline{\omega}}(-\overline{\omega}\cdot
A_{1})$ and $\Psi_{2}=N^{-1}_{\overline{\omega}}(-\overline{\omega}\cdot
A_{2})$ in such away that we have
$$\Psi_{2}-\overline{\Psi_{1}}=N^{-1}_{\overline{\omega}}(-(-\overline{\omega}\cdot A))=-N_{\overline{\omega}}^{-1}(-\overline{\omega}\cdot A).$$
Then, we infer from (\ref{mag3}) that
$$\begin{array}{lll}
\displaystyle\int_{Q}A(x)\cdot(\rho_{2}+\overline{\rho_{1}})e^{-ix\cdot
\xi}e^{i(\phi_{2}-\overline{\phi_{1}})}dxdt&=&
 J_{1}(\xi,\sigma)+J_{2}(\xi,\sigma)+J_{3}(\xi,\sigma),
\end{array}$$
where we have set
$$J_{1}(\xi,\sigma)=2\sigma\displaystyle\int_{Q}\overline{\omega}\cdot A(x)e^{-ix\cdot
\xi}e^{i(\Psi_{2}-\overline{\Psi_{1}})}\,dx\,dt,$$
$$J_{2}(\xi,\sigma)=-2\sigma \displaystyle\int_{Q}\overline{\omega}\cdot A(x)e^{-ix\cdot \xi}\displaystyle\para{e^{i(\Psi_{2}-\overline{\Psi_{1}})}-
e^{i(\phi_{2}-\overline{\phi_{1}})}}\,dx\,dt,$$
and 
$$J_{3}(\xi,\sigma)=-2\sigma \displaystyle\para{1-\displaystyle\sqrt{1-\displaystyle|\xi|^{2}/4\sigma^{2}}}\displaystyle\int_{Q}\omega_{\Re}\cdot A(x)e^{-ix\cdot \xi}e^{i(\phi_{2}-\overline{\phi_{1}})}\,dx\,dt.$$
Using Lemma \ref{Lm2.2}, one can see that
$$\begin{array}{lll}
J_{1}(\xi,\sigma)&=
&2\sigma T\displaystyle\int_{\Omega}\overline{\omega}\cdot A(x)e^{iN^{-1}_{\overline{\omega}}(-\overline{\omega}\cdot (-A))} e^{-ix\cdot \xi}\,dx\\
&=&2\sigma T\, \displaystyle\int_{\Omega}\overline{\omega}\cdot A(x) e^{-ix\cdot \xi}\,dx\,dt.
\end{array}$$
Now it remains to  upper bound  the absolute value of
$J:=J_{2}+J_{3}$. We start by inserting $e^{i(\Psi_{2}-\overline{\phi_{1}})}$
into $J_{2}(\xi,\sigma)$, getting 
$$\begin{array}{lll}
 J_{2}(\xi,\sigma)
&=&-2\sigma T \displaystyle\int_{\Omega}\overline{\omega}\cdot A(x)e^{-ix\cdot \xi}\displaystyle\para{e^{i\Psi_{2}}\displaystyle\para{e^{-i\overline{\Psi_{1}}}
-e^{-i\overline{\phi_{1}}}}
+e^{-i\overline{\phi_{1}}}\displaystyle\para{e^{i\Psi_{2}}-e^{i\phi_{2}}}} \,dx.
\end{array}$$
Further, as $N^{-1}_{\omega}(-\omega\cdot A)$ depends continuously on $\omega$, 
according to Lemma $2.4$ in \cite{[L]}, we get for all $|\xi|\leq2\sigma$
$$|J_{2}(\xi,\sigma)|\leq C_{T}\sigma \para{|\overline{\omega}-\overline{\omega_{1}^{*}}|+|\overline{\omega}-\omega_{2}^{*}
|}.$$
Hence, as  $1-\sqrt{1-|\xi|^{2}/4\sigma^{2}}\leq |\xi|^{2}/4\sigma^{2}$
for all $|\xi|\leq2\sigma$, we deduce from (\ref{Eq8}), (\ref{Eq9}) and the above inequality  that
$$|J_{2}(\xi,\sigma)|\leq C_{T}\para{\sigma \frac{|\xi|^{2}}{4\sigma^{2}}+|\xi|}\leq C_{T}\,|\xi|.$$
Arguing in the same way, we find that $ |J_{3}(\xi,\sigma)|\leq C_{T} |\xi|,
$ for some  positive constant $C_{T}$ which is independent of $\xi$ and
$\sigma$.
\end{proof}
\subsection{Estimating the Fourier transform of the magnetic field}
We aim to relate the Fourier transform of the magnetic field
$d{\alpha_{A_{1}}}-d{\alpha_{A_{2}}}$ to the measurement
$\Lambda_{A_{1},q_{1}}-\Lambda_{A_{2},q_{2}}$. To this end, we introduce the
following notation: we put
$${a}_{k}(x)=(A_{1}-A_{2})(x)\cdot e_{k}=A(x)\cdot e_{k},$$
where $(e_{k})_{k}$ is the canonical basis of $\R^{n}$, and
\begin{equation}\label{def}
\sigma_{j,k}(x)=\frac{\p {a}_{k}}{\p x_{j}}(x)-\frac{\p {a_{j}}}{\p
x_{k}}(x), \,\,\,\,\,\,j,\,k=1,...,n.
\end{equation}
We recall that the Green formula for the magnetic Laplacian
\begin{equation}\label{green}
\int_{\Omega} (\Delta_{A}u \overline{v}-u\overline{\Delta_{A} v})
\,dx=-\int_{\Gamma} \Big( (\p_{\nu}+i\nu.A
)u\overline{v}-u\overline{(\p_{\nu}+iA.\nu)v}\Big)\,d\sigma_{x},
\end{equation}
holds  for any $u,\,v\in H^{1}(\Omega)$ such that $\Delta u,\,\Delta v\in
L^{2}(\Omega)$. Here $d\sigma_{x}$ is the Euclidean surface measure on
$\Gamma$.  We  estimate the Fourier transform of $\sigma_{j,k}$ as
follows.
\begin{Lemm}\label{Lemme3.3}
Let $\xi\in\R^{n}$ and  $\sigma>\max(\sigma_{0},|\xi|/2)$, where
$\sigma_{0}$ is as in Lemma \ref{Prop3.1}. Then we have
$$<\xi>^{-1}|\widehat{\sigma}_{j,k}(\xi)|\leq C\para{e^{C\sigma}\|\Lambda_{A_{2},q_{2}}-\Lambda_{A_{1},q_{1}}\|+\frac{1}{\sigma}+\frac{|\xi|}{|\sigma|} },$$
where $C$ is independent of $\xi$ and $\sigma$.
\end{Lemm}
\begin{proof}{} First, for $\sigma>\sigma_{0}$, Lemma \ref{Prop3.1} guarantees the
existence of a geometrical optic solution $u_{2}$, of the form
$$u_{2}(x,t)=e^{-ix.\rho_{2}}(e^{i\phi_{2}(x)}+w_{2}(x,t))$$
to the magnetic Schr\"odinger equation
\begin{equation}
\left\{
  \begin{array}{ll}
   (i\p_{t}+\Delta_{A_{2}}+q_{2}(x,t))u_{2}(x,t)=0, & \mbox{in}\,\,Q, \\
  u_{2}(x,0)=u_{0} , & \mbox{in}\,\,\Omega, \\
  \end{array}
\right.
\end{equation}
where $\rho_{2}$ is given by (\ref{Eq9}).
Let us  denote by $f_{\sigma}:=u_{2|\Sigma}$. We  consider a solution $v$ to
the following non homogeneous boundary value problem
\begin{equation}\label{u}
\left\{
  \begin{array}{ll}
    (i\p_{t}+\Delta_{A_{1}}+q_{1}(x,t))v=0, & \mbox{in}\,\,Q, \\
    v(.,0)=u_{2}(.,0)=u_{0}, & \mbox{in}\,\,\Omega, \\
    v=u_{2}=f_{\sigma}, & \mbox{on}\,\Sigma.
  \end{array}
\right.
\end{equation}
Then, $u=v-u_{2}$ is a solution to the following homogenous boundary value
problem for the magnetic Schr\"odinger equation
$$
\left\{
  \begin{array}{ll}
    (i\p_{t}+\Delta_{A_{1}}+q_{1}(x,t))u=2iA\cdot \nabla u_{2}+h(x,t) u_{2}, & \mbox{in}\,Q, \\
    u(x,0)=0, & \hbox{in}\, \Omega,\\
    u(x,t)=0, & \mbox{on}\,\Sigma,
  \end{array}
\right.
$$
where
$$A=A_{1}-A_{2},\quad q=q_{1}-q_{2}\quad \mbox{and}\quad h=i \,\mbox{div}
A-(|A_{1}|^{2}-|A_{2}|^{2})+q.$$ On the other hand, with reference to Lemma \ref{Prop3.1} we consider a solution
$u_{1}$ to  the magnetic Shr\"odinger equation (\ref{Eq6}), associated with 
the potentials $A_{1}$ and $q_{1}$, of the form
$$u_{1}(x,t)=e^{-ix.\rho_{1}}(e^{i\phi_{1}(x)}+w_{1}(x,t)),$$
where $\rho_{1}$ is given by (\ref{Eq8}).
 Integrating by parts in the following integral, and  using the Green Formula  (\ref{green}), we get
\begin{eqnarray}\label{ap} \displaystyle\int_{Q}(i\p_{t}+\Delta_{
A_{1}}+q_{1})u\overline{u_{1}}dxdt\!\!\!&=&\!\!\!\!\!\!\displaystyle\int_{Q}2iA\cdot\nabla
u_{2}\overline{u_{1}}dxdt
+\displaystyle\int_{Q}\!\!\Big(i\mbox{div}A-(|A_{1}|^{2}-|A_{2}|^{2})+q\Big)u_{2}\overline{u_{1}} dx dt\cr
&=&i\displaystyle\int_{\Omega} u(.,T)\overline{u_{1}}(.,T)\,dx-\displaystyle\int_{\Sigma}(\p_{\nu}+iA_{1}.\nu)u \overline{u_{1}}\,d\sigma_{x}\,dt.
\end{eqnarray}
This  entails that
$$\begin{array}{lll}
\displaystyle\int_{Q}2iA\cdot \nabla u_{2} \overline{u_{1}}dx\,dt&=-&i\displaystyle\int_{\Omega} (\Lambda_{A_{2},q_{2}}^{1}-\Lambda_{A_{1},q_{1}}^{1})(g)\overline{u_{1}}(.,T)\,dx
+\displaystyle\int_{\Sigma}(\Lambda_{A_{2},q_{2}}^{2}
-\Lambda_{A_{1},q_{1}}^{2})(g) \overline{u_{1}}\,d\sigma_{x}\,dt\\
&&-\displaystyle\int_{Q}\!\!\Big(i\mbox{div}A-(|A_{1}|^{2}-|A_{2}|^{2})+q\Big)u_{2}\overline{u_{1}} dx dt,
\end{array}$$
 where $g=(u_{2|t=0},u_{2|\Sigma})$.
Upon applying the Stokes formula and using the fact that $A_{|\Gamma}=0$, 
we get
\begin{eqnarray}\label{Eq11}
\displaystyle\int_{Q}\!\! i A\!\cdot\!\big( \overline{u_{1}}\nabla u_{2}-u_{2}\nabla \overline{u_{1}}\big)dxdt
\!\!\!\!\!&=&\!\!\!\!\!-i\displaystyle\int_{\Omega}\!\! \displaystyle\para{\Lambda_{A_{2},q_{2}}^{1}\!-\!
\Lambda_{A_{1},q_{1}}^{1}}(g)\,\overline{u_{1}}(.,T)\,dx+\!\displaystyle\int_{\Sigma}\!\!
\displaystyle\para{\Lambda_{A_{2},q_{2}}^{2}\!-\!\Lambda_{A_{1},q_{1}}^{2}}(g)\,\overline{u_{1}}d\sigma_{x}dt\cr
&&+\displaystyle\int_{Q}\Big(|A_{1}|^{2}-|A_{2}|^{2}+q\Big)u_{2}\overline{u_{1}}\,dx\,dt.
\end{eqnarray}
This, Lemma \ref{Lm4.2} and Lemma \ref{Lm4.3}, yield
$$\Big|\int_{\Omega}\overline{\omega}.A(x)e^{-ix.\xi}\,dx\Big|\leq \frac{C_{T}}{\sigma}\Big( \|\Lambda_{A_{2},q_{2}}-
\Lambda_{A_{1},q_{1}}\|\,\|g\|_{H^{2}(\Omega)\times
H^{2,1}(\Sigma)}\|\phi\|_{L^{2}(\Sigma)\times L^{2}(\Omega)}+C+|\xi|
\Big),$$ where $\phi=(\overline{u_{1}}_{|\Sigma},\overline{u_{1}}_{t=T})$.
Here we used the fact that $ \|u_{2}\overline{u_{1}}\|_{L^{1}(Q)}\leq C_{T} $, 
for $\sigma$ sufficiently large.
 Hence, bearing in mind that
$$\|g\|_{H^{2}(\Omega)\times H^{2,1}(\Sigma)}\leq C e^{C\sigma},\,\,\,\,\,\,\,\mbox{and}\,\,\,\,\,\|\phi\|_{L^{2}(\Sigma)\times L^{2}(\Omega) }\leq C e^{C\sigma},$$
we get for $\sigma>|\xi|/2$,
\begin{equation}\label{aj4}
\Big|\int_{\Omega}\overline{\omega}\cdot A(x) e^{-ix\cdot\xi}\,dx\Big|\leq C\para{e^{C\sigma}\|\Lambda_{A_{2},q_{2}}-\Lambda_{A_{1},q_{1}}\|+\frac{1}{\sigma}
+\frac{|\xi|}{\sigma}}.
\end{equation}
Arguing as in the derivation of (\ref{aj4}), we prove by replacing
$\overline{\omega}$ by $-\omega$, that 
\begin{equation}\label{aj5}
\Big|\int_{\Omega}-\omega\cdot A(x)e^{-ix\cdot\xi}\,dx\Big|\leq C \para{e^{C\sigma}\|\Lambda_{A_{2},q_{2}}-\Lambda_{A_{1},q_{1}}\|+\frac{1}{\sigma}+\frac{|\xi|}{\sigma}}.
\end{equation}
Thus, choosing
$\omega_{\Im}=\frac{\xi_{j}e_{k}-\xi_{k}e_{j}}{|\xi_{j}e_{k}-\xi_{k}e_{j}|},$
 multiplying (\ref{aj4}) and (\ref{aj5}) by $|\xi_{j}e_{k}-\xi_{k}e_{j}|$, and adding the obtained inequalities together, we find that
$$\Big|\int_{\Omega}e^{-ix\cdot\xi}\para{\xi_{j}\tilde{a}_{k}(x)-\xi_{k}\tilde{a_{j}}(x)}\,dx|\leq C\,|\xi_{j}e_{k}-e_{j}\xi_{k}|\para{e^{C\sigma}\|\Lambda_{A_{2},q_{2}}-\Lambda_{A_{1},q_{1}}\|+\frac{1}{\sigma}+\frac{|\xi|}{\sigma} }.$$
From this and (\ref{def}) we deduce that
$$|\widehat{\sigma}_{j,k}(\xi)|\leq C <\xi>\para{e^{C\sigma}\|\Lambda_{A_{2},q_{2}}-\Lambda_{A_{1},q_{1}}\|+\frac{1}{\sigma}+\frac{|\xi|}{\sigma} },\,\,\,\,\,\,\,j,\,k\in
\mathbb{N}.$$ This ends  the proof.
\end{proof}
\subsection{Stability estimate}
Armed with Lemma \ref{Lemme3.3}, we are now in position to complete the proof
of the stability estimate for the magnetic field. To do so, we first need to
 bound the $H^{-1}(\R^{n})$ norm of
$d\alpha_{A_{1}}-d\alpha_{A_{2}}$. In light of the above reasoning , this can be achieved by taking
$\sigma>R>0$ and decomposing the $H^{-1}(\R^{n})$ norm of $\sigma_{j,k}$ as 
$$
\|\sigma_{j,k}\|^{2}_{H^{-1}(\R^{n})}=\displaystyle\int_{|\xi|\leq
R}|\widehat{\sigma}_{j,k}(\xi)|^{2}<\xi>^{-2}\,d\xi
+\displaystyle\int_{|\xi|>R}|\widehat{\sigma}_{j,k}(\xi)|^{2}
<\xi>^{-2}\,d\xi.$$
Then, we have
$$\|\sigma_{j,k}\|^{2}_{H^{-1}(\R^{n})}\leq C \Big[R^{n}\|<\xi>^{-1}\widehat{\sigma}_{j,k}\|^{2}_{L^{\infty}(B(0,R))}+\frac{1}{R^{2}}
\|\sigma_{j,k}\|_{L^{2}(\R^{n})}^{2}\Big],$$
 which entails that
$$\|\sigma_{j,k}\|^{2}_{H^{-1}(\R^{n})}\leq C\Big[R^{n}\para{e^{C\sigma}\|\Lambda_{A_{2},q_{2}}
-\Lambda_{A_{1},q_{1}}\|^{2}+\frac{1}{\sigma^{2}}+\frac{R^{2}}{\sigma^{2}}}+\frac{1}{R^{2}}\Big],
$$
 by Lemma \ref{Lemme3.3}. The next step is to choose  $R>0$ in such away
$\frac{R^{n+2}}{\sigma^{2}}=\frac{1}{R^{2}}$. In this case we get for
$\sigma>\max(\sigma_{0},|\xi|/2)$, that
\begin{eqnarray}\label{aj6}
\|\sigma_{j,k}\|^{2}_{H^{-1}(\R^{n})}&\leq& C\para{\sigma^{\frac{2n}{n+4}}e^{C\sigma}\|\Lambda_{A_{2},q_{2}}-\Lambda_{A_{1},q_{1}}\|^{2}+\sigma^{\frac{-4}{n+4}}}\cr
&\leq& C \para{e^{C_{0}\sigma}\|\Lambda_{A_{2},q_{2}}-\Lambda_{A_{1},q_{1}}\|^{2}+\frac{1}{\sigma^{\mu}}},
\end{eqnarray}
where $\mu\in (0,1)$. Thus, assuming that $\|\Lambda_{A_{2},q_{2}}-\Lambda_{A_{1},q_{1}}\|\leq c=e^{-C_{0}\max\para{\sigma_{0},|\xi|/2}}$, and taking  
$\sigma=\frac{1}{C_{0}}|\log\|\Lambda_{A_{2},q_{2}}-\Lambda_{A_{1},q_{1}}\||$ in (\ref{aj6}), we get that 
$$\|\sigma_{j,k}\|_{H^{-1}(\R^{n})}\leq C
\para{\|\Lambda_{A_{2},q_{2}}-\Lambda_{A_{1},q_{1}}\|^{1/2}+|\log\|\Lambda_{A_{2},q_{2}}-\Lambda_{A_{1},q_{1}}\||^{-\mu'}},$$
for some positive $\mu'\in(0,1)$. Since the above estimate remains true when
 $\|\Lambda_{A_{2},q_{2}}-\lambda_{A_{1},q_{1}}\|\geq c$, as we have 
$$\|\sigma_{j,k}\|_{H^{-1}(\R^{n})}\leq \frac{2 M}{c^{1/2}}c^{1/2}\leq \frac{2M}{c^{1/2}}\|\Lambda_{A_{2},q_{2}}-\Lambda_{A_{1},q_{1}}\|^{1/2}, $$
we have obtained that
$$\begin{array}{lll}
\|d\alpha_{A_{1}}-d{\alpha_{A_{2}}}\|_{H^{-1}(\Omega)}
&\leq& C\para{\|\Lambda_{A_{2},q_{2}}-\Lambda_{A_{1},q_{1}}\|^{1/2}+|\log\|\Lambda_{A_{2},q_{2}}-\Lambda_{A_{1},q_{1}}\||^{-\mu'}}.
\end{array}$$
In order to complete the proof of the theorem, we consider $\delta>0$ such
that $\alpha:=s-1=\frac{n}{2}+2\delta$, use Sobolev's embedding theorem and
we find
$$\begin{array}{lll}
\|d\alpha_{A_{1}}-d\alpha_{A_{2}}\|_{L^{\infty}(\Omega)}&\leq& C \|d\alpha_{A_{1}}-d\alpha_{A_{2}}\|_{H^{\frac{n}{2}+\delta}(\Omega)}\\
&\leq& C \|d\alpha_{A_{1}}-d\alpha_{A_{2}}\|_{H^{-1}(\Omega)}^{1-\beta}\|d\alpha_{A_{1}}-d\alpha_{A_{2}}\|_{H^{s-1}(\Omega)}^{\beta}\\
&\leq&C\para{\|\Lambda_{A_{2},q_{2}}-\Lambda_{A_{1},q_{1}}\|^{1/2}+|\log \|\Lambda_{A_{2},q_{2}}-\Lambda_{A_{1},q_{1}}\||^{-\mu}}^{1-\beta},
\end{array}$$
by interpolating with $\beta\in (0,1)$. This completes the proof of Theorem \ref{Thm1}.

This theorem is a key ingredient in the proof of the result of the next section.
\section{Stability result  for the electric potential}\label{Sec4}%
This section contains the proof of Theorem \ref{Thm2}. Using the geometric optics solutions constructed in Section\ref{Sec2}, we will prove with the aid of  the stability estimate obtained for the magnetic field,
that the time-dependent electric potential depends stably on the
Dirichlet-to-Neuamnn map $\Lambda_{A,q}$.

To do this, we should normally  apply the Hodge decomposition to
$A=A_{1}-A_{2}=A'+\nabla\varphi$ and use this estimate
\begin{equation}\label{o}
\|A'\|_{W^{1,p}(\Omega)}\leq C \|\mbox{curl} \,A'\|_{L^{p}(\Omega)}.
\end{equation}
that holds for any $p>n$ (see Appendix B).
But in this paper, since $u_{0}$ is not frozen to zero, we don't have
invariance under Gauge transformation, so will further assume that $A$
is divergence free
in such a way that the estimate (\ref{o}) holds for $A'=A$.

For a fixed  $y\in B(0,1)$, we consider solutions $u_{j}$ to the Schr\"odinger
equation of the form (\ref{Eq10}) with $\rho_{j}= \sigma\omega_{j}^{*}+y$, where
$\xi\in \R^{n}$ and $\omega\in\mathbb{S}^{n-1}$ are as in Section \ref{Sec3}, 
and $w_{j}^{*}$, $j=1,2$, are given by (\ref{Eq8}) and (\ref{Eq9}).

In contrast to Section \ref{Sec3}, $y$ is no longer equal to zero, as we need 
to estimate the Fourier transform of $q$ with
respect to $x$ and $t$.
\subsection{An identity for the electric potential }
Let us first establish the following identity for the electric
potential.
\begin{Lemm}\label{Lem3.3}
Let $u_{j}$ be the solutions given by (\ref{Eq10}) for $j=1,2$. For all  $\sigma\geq
\sigma_{0}$ and $\xi\in\R^{n}$ such that $|\xi|<2\sigma$, we have the
following identity
$$\int_{Q}q(x,t) u_{2}\overline{u_{1}}\,dx\,dt=\int_{Q}q(x,t) e^{-i(2y.\xi t+x.\xi)}\,dx\,dt+P_{1}(\xi,y,\sigma)+P_{2}(\xi,y,\sigma),$$
where  $P_{1}(\xi,y,\sigma)$ and $P_{2}(\xi,y,\sigma)$ satisfy the estimates 
$$|P_{1}(\xi,y,\sigma)|\leq C\para{\|A\|_{L^{\infty}(\Omega)}+\frac{|\xi|}{\sigma}},\,\,\,\,\,\,\,|P_{2}(\xi,y,\sigma)|\leq \frac{C}{\sigma}\,.$$
Here $\sigma_{0}$ is as in Lemma \ref{Prop3.1} and $C$ is independent of
$\sigma,\,y,$ and $\,\xi$.
\end{Lemm}
\begin{proof}{}
In light of (\ref{Eq8}), (\ref{Eq9}) and (\ref{Eq10}), a direct calculation
gives us
\begin{eqnarray}\label{X}
u_{2}\overline{u_{1}}&=&e^{-i\Big((\rho_{2}.\rho_{2}-\overline{\rho_{1}.\rho_{1}})t+x.(\rho_{2}-\overline{\rho_{1}})\Big)}
\Big(e^{i(\phi_{2}-\overline{\phi_{1}})}+e^{-i\overline{\phi}_{1}}w_{2}+e^{i\phi_{2}}\overline{w_{1}}+w_{2}\overline{w}_{1}\Big)\cr
&=&e^{-i(2y.\xi t+x.\xi)}e^{-i(\overline{\phi_{1}}-\phi_{2})} +e^{-i(2y.\xi t+x.\xi)}\para{e^{-\overline{\phi_{1}}}w_{2}
+e^{i\phi_{2}}\overline{w_{1}}+w_{2}\overline{w_{1}}},
\end{eqnarray}
which yields
\begin{eqnarray}\label{Eq3.19}
\int_{Q}q(x,t)u_{2}\overline{u_{1}}\,dx\,dt=\int_{Q}q(x,t)e^{-i(2y.\xi t +x.\xi)}\,dx\,dt+P_{1}(\xi,y,\sigma)+P_{2}(\xi,y,\sigma),
\end{eqnarray}
where we have set
$$\begin{array}{lll}P_{1}(\xi,y,\sigma)&=&\displaystyle\int_{Q}q(x,t)e^{-i(2y\cdot\xi
t+x\cdot\xi)}e^{-i\overline{\phi}_{1}}\para{e^{i\phi_{2}}-e^{i\overline{\phi}_{1}}}\,dx\,dt,\\
P_{2}(\xi,y,\sigma)&=&\displaystyle\int_{Q}q(x,t)e^{(-i2y\cdot\xi t+x.\xi)}\para{e^{-i\overline{\phi}_{1}}w_{2}+e^{i\phi_{2}}\overline{w_{1}}+w_{2}\overline{w_{1}}}\,dx\,dt.
\end{array}$$
Recalling that  $\phi_{j}=N^{-1}_{\omega_{j}^{*}}(-\omega_{j}^{*}\cdot A_{j})$, for $j=1,\,2$, we deduce from the definition of $P_{1}$ that  
$$\begin{array}{lll}
|P_{1}(\xi,y,\sigma)|
&\leq& C\Big(\|e^{i N^{-1}_{\omega_{2}^{*}}(-\omega_{2}^{*}\cdot A_{2})}-e^{i N^{-1}_{\omega_{2}^{*}}(-\omega_{2}^{*}\cdot A_{1})}\|_{L^{\infty}(\Omega)}+
\|e^{i N_{\omega_{2}^{*}}^{-1}(-\omega_{2}^{*}\cdot A_{1})}-e^{iN_{\overline{\omega_{1}}^{*}}^{-1}(-\overline{\omega_{1}}^{*}\cdot A_{1})}\|_{L^{\infty}(\Omega)}\Big),
\end{array}$$
with $C>0$ is depending on $T$, $M$,  $\Omega$ and $\|A_{1}\|$. Using the
continuity of $N_{\omega}^{-1}(-\omega\cdot A)$ with respect to $\omega$ (see
Lemma  2.4 in \cite{[L]}), we get that
$$\begin{array}{lll}
|P_{1}(\xi,y,\sigma)|&\leq& C\para{ \|{N_{\omega_{2}}^{*}}^{-1}(-\omega_{2}^{*}.A_{2})-N_{\omega_{2}^{*}}^{-1}(-\omega_{2}^{*}.A_{1})\|_{L^{\infty}(\Omega)}
+|\omega^{*}_{2}-\overline{\omega}_{1}^{*}|}\\
&\leq&C\para{\|A\|_{L^{\infty}(\Omega)}+\displaystyle\frac{|\xi|}{\sigma}}.
\end{array}$$
 On the other hand, from Cauchy Schwarz
inequality, Lemma \ref{Lm2.1} and (\ref{Equation 2.17}), we get
$$\begin{array}{lll}
 |P_{2}(\xi,y,\sigma)|&\leq& C\para{\|w_{2}\|_{L^{2}(Q)}\|e^{-i\overline{\phi_{1}}}\|_{L^{2}(Q)}+\|e^{i\phi_{2}}\|_{L^{2}(Q)}\|\overline{w_{1}}\|_{L^{2}(Q)}
 +\|w_{2}\|_{L^{2}(Q)}\|\overline{w_{1}}\|_{L^{2}(\Omega)}}\\
 &\leq& \displaystyle\frac{C}{\sigma}.
 \end{array}$$
This completes the proof of Lemma \ref{Lem3.3}.
\end{proof}
\subsection{Estimate of the Fourier transform}
In view of relating the Fourier transform of the electric potential
$q=q_{1}-q_{2}$ to  $\Lambda_{A_{1},q_{1}}-\Lambda_{A_{2},q_{2}}$, we first
establish the following auxiliary result
\begin{Lemm}\label{Lm3.4}
For any $\sigma \geq \sigma_{0}$ and $\xi\in \R^{n}$ such that
$|\xi|<2\sigma$, we have the following estimate 
$$
|\widehat{q}(\xi,2y.\xi)|\leq C\Big(  e^{C\sigma}\|\Lambda_{A_{2},q_{2}}-\Lambda_{A_{1},q_{1}}\|+e^{C\sigma}\|d{\alpha_{A_{1}}}-d{\alpha_{A_{2}}}\|_{L^{\infty}(\Omega)}+\frac{|\xi|}{\sigma} +\frac{1}{\sigma} \Big),
$$
for some $C$ that is independent of $|\xi|$ and $\sigma.$
\end{Lemm}
\begin{proof}{}
First, for $\sigma>\sigma_{0}$, Lemma \ref{Prop3.1} guarantees the existence
of a geometrical optics solution $u_{2}$ of the form
$$u_{2}(x,t)=e^{-i\para{(\rho_{2}.\rho_{2})t+x.\rho_{2}}}(e^{i\phi_{2}(x)}+w_{2}(x,t)),$$
to the magnetic Schr\"odinger equation
\begin{equation}
\left\{
  \begin{array}{ll}
  (i\p_{t}+\Delta_{A_{2}}+q_{2}(x,t))u_{2}(x,t)=0, &\mbox{in}\,Q,\\
  u_{2}(x,0)=u_{0}, &\mbox{in}\,\,\Omega,
   \end{array}
\right. 
\end{equation}
where $\rho_{2}$ is given by (\ref{Eq9}) and $w_{2}(x,t)$ satisfies
\begin{equation}\label{Eq3.28}
\sigma\|w_{2}\|_{H^{2}(0,T,H^{1}(\Omega))}+\|w_{2}\|_{L^{2}(0,T,H^{2}(\Omega))}\leq C
\end{equation}
Let us  denote by $f_{\sigma}:=u_{2|\Sigma}$. We  consider a solution $v$ to
the following non homogeneous boundary value problem
\begin{equation}\label{u}
\left\{
  \begin{array}{ll}
    (i\p_{t}+\Delta_{A_{1}}+q_{1}(x,t))v=0, & \mbox{in}\,\,Q, \\
    v(.,0)=u_{2}(.,0)=u_{0}, & \mbox{in}\,\,\Omega, \\
    v=u_{2}=f_{\sigma}, & \mbox{on}\,\Sigma.
  \end{array}
\right.
\end{equation}
Denote $u=v-u_{2}$, then $u$ is a solution to the following homogenous
boundary value problem for the magnetic Schr\"odinger equation
$$
\left\{
  \begin{array}{ll}
    (i\p_{t}+\Delta_{A_{1}}+q_{1}(x,t))u=2iA\cdot \nabla u_{2}+h(x,t) u_{2}, & \mbox{in}\,Q, \\
    u(x,0)=0, & \hbox{in}\, \Omega,\\
    u(x,t)=0, & \mbox{on}\,\Sigma,
  \end{array}
\right.
$$
where we recall that
$$A=A_{1}-A_{2},\quad q=q_{1}-q_{2}\quad \mbox{and}\quad h=i \,\mbox{div}
A-(|A_{1}|^{2}-|A_{2}|^{2})+q.$$
On the other hand, we consider a solution $u_{1}$ of the magnetic
Shr\"odinger equation (\ref{Eq6}) corresponding to the potentials $A_{1}$ and
$q_{1}$, of the form
$$u_{1}(x,t)=e^{-i\para{(\rho_{1}.\rho_{1})t+x.\rho_{1}}}(e^{i\phi_{1}(x)}+w_{1}(x,t)),$$
where $\rho_{1}$ is given by (\ref{Eq8}) and  $w_{1}(x,t)$ satisfies
\begin{equation}\label{eq3.29}
\sigma\|w_{1}\|_{H^{2}(0,T,H^{1}(\Omega))}+\|w_{1}\|_{L^{2}(0,T,H^{2}(\Omega))}\leq C.
\end{equation}
 Integrating by parts and  using the Green Formula (\ref{green}), we get
$$\begin{array}{lll}
\displaystyle\int_{Q}q(x,t) u_{2}\overline{u_{1}}\,dx\,dt&=&i\displaystyle\int_{\Omega}(\Lambda_{A_{2},q_{2}}^{1}-\Lambda_{A_{1},q_{1}}^{1})(g)\overline{u_{1}}(.,T)\,dx
-\displaystyle\int_{\Sigma}(\Lambda_{A_{2},q_{2}}^{1}-\Lambda_{A_{1},q_{1}}^{2})(g)\overline{u_{1}}\,d\sigma_{x}\,dt\\
&&+\displaystyle\int_{Q}i A(x)\cdot (\overline{u_{1}}\nabla u_{2}-u_{2}\nabla \overline{u_{1}})\,dx\,dt
-\int_{Q}(|A_{1}|^{2}-|A_{2}|^{2})u_{2}\overline{u_{1}}\,dx\,dt,
\end{array}
$$
where $g=(u_{2|t=0},u_{2|\Sigma}).$  To bring the Fourier transform of
$q$ out of the above identity, we  extend $q$ by zero outside the cylindrical domain $Q$, we use Lemma
\ref{Lem3.3} and take to account that
$$\|u_{2}\overline{u_{1}}\|_{L^{1}(Q)}\leq C,\,\,\,\mbox{and}\,\,\,\,\,\|\overline{u_{1}}\nabla u_{2}\|_{L^{1}(Q)}+\|u_{2}\nabla \overline{u_{1}}\|_{L^{1}(Q)}\leq C\sigma,$$
and  get
$$\begin{array}{lll}
|\widehat{q}(\xi,2y\cdot\xi)|\leq C \Big(\|\Lambda_{A_{2},q_{2}}-\Lambda_{A_{1},q_{1}}\|\|g\|_{H^{2}(\Omega)\times H^{2,1}(\Sigma)}
\|\phi\|_{L^{2}(\Sigma)\times L^{2}(\Omega)}+C\sigma\|A\|_{L^{\infty}(\Omega)}+\displaystyle\frac{|\xi|}{\sigma}+\displaystyle\frac{1}{\sigma}\Big),
\end{array}$$
where $\phi=(\overline{u_{1}}_{|\Sigma},\,\overline{u_{1}}_{|t=T})$.
 Now, bearing in mind that
$$\|g\|_{H^{2}(\Omega)\times H^{2,1}(\Sigma)}\leq C e^{C\sigma},\,\,\,\,\,\,\,\mbox{and}\,\,\,\,\,\|\phi\|_{L^{2}(\Sigma)\times L^{2}(\Omega)}\leq C e^{C\sigma},$$
we get for all $\xi\in\R^{n}$ such that $|\xi|<2\sigma$ and for all $y\in
B(0,1)$,
\begin{equation}
|\widehat{q}(\xi,2y.\xi)|\leq C\Big(  e^{C\sigma}\|\Lambda_{A_{2},q_{2}}-\Lambda_{A_{1},q_{1}}\|+e^{C\sigma}\|A\|_{L^{\infty}(\Omega)}+\frac{|\xi|}{\sigma} +\frac{1}{\sigma} \Big).
\end{equation}
Finally, using the fact that $\|A\|_{W^{1,\infty}(\Omega)}\leq C 
\|\mbox{curl}\,A\|_{L^{\infty}(\Omega)},$  (see Lemma \ref{rot} in Appendix B), we obtain the desired result.
\end{proof}
We are now in position to estimate $\widehat{q}(\xi,\tau)$ for all
$(\xi,\tau)$ in the following set
$$E_{\alpha}=\{(\xi,\tau)\in (\R^{n}\setminus\{0\})\times \R,\,\,|\xi|<2\alpha,\,\,\,|\tau|<2|\xi|\},$$
for any fixed $0<\alpha<\sigma$.
\begin{Lemm}\label{Lm3.5}
Suppose that the conditions of Lemma \ref{Lm3.4} are satisfied. Then we have
for all $(\xi,\tau)\in E_{\alpha}$,
\begin{equation}\label{ball}
|\widehat{q}(\xi,\tau)|\leq C\Big(
e^{C\sigma}\|\Lambda_{A_{2},q_{2}}-\Lambda_{A_{1},q_{1}}\|
+e^{C\sigma}\|d{\alpha_{A_{1}}}-d{\alpha_{A_{2}}}\|_{L^{\infty}(\Omega)}+\frac{\alpha}{\sigma}
+\frac{1}{\sigma} \Big).
\end{equation}
 Here $C$ is independent of $|\xi|$ and $\sigma$.
\end{Lemm}
\begin{proof}{}
Fix $(\xi,\tau)\in E_{\alpha},$ and set  $y=\frac{\tau}{2|\xi|^{2}}\cdot
\xi,$ in such away  that $y\in B(0,1)$ and $2y\cdot\xi=\tau$. Since
$\alpha<\sigma$ we have $|\xi|<2\alpha<2\sigma$. Hence,
  Lemma \ref{Lm3.4} yields the desired result.
\end{proof}
\subsection{Stability estimate}
In order to complete the proof of the stability estimate for
the electric potential, we use an argument for analytic functions
proved in \cite{[new]} (see also \cite{[A],[V]}). For $\gamma\in
\mathbb{N}^{n+1}$, we put $|\gamma|=\gamma_{1}+...+\gamma_{n+1}.$ We have the
following statement that claims conditional stability for the analytic continuation.
\begin{Lemm}\label{Lm3.6}
Let $O$ be a non empty open set of $B(0,1)$ and let $F$ be an analytic
function in $B(0,2)$, obeying
$$\|\p^{\gamma}F\|_{L^{\infty}(B(0,2))}\leq \frac{M |\gamma|}{\eta^{|\gamma|}},\,\,\,\,\,\,\forall \gamma\in \mathbb{N}^{n+1}$$
for some $M>0$ and $\eta>0$. Then we have
$$\|F\|_{L^{\infty}(B(0,1))}\leq (2M)^{1-\mu}\|F\|^{\mu}_{L^{\infty}(O)},$$
where $\mu\in(0,1)$ depends on $n$, $\eta$ and $|O|$.
\end{Lemm}
 We refer to Lavrent'ev \cite{[Lav]} for classical results for this type.  For
fixed $0<\alpha<\sigma$, let us set
$$F_{\alpha}(\xi,\tau)=\widehat{q}(\alpha(\xi,\tau)),\,\,\,\,\,\,\,\,(\xi,\tau)\in\R^{n+1}.$$
It is easily seen that $F_{\alpha}$ is analytic and that 
$$\begin{array}{lll}
|\p^{\gamma}F_{\alpha}(\xi,\tau)|=|\p^{\gamma}\widehat{q}(\alpha(\xi,\tau))|&=&\Big|\p^{\gamma}\displaystyle\int_{\R^{n+1}}q(x,t) e^{-\alpha (x,t).(\tau,\xi)}\,dx\,dt\Big|\\
&=&\Big|  \displaystyle\int_{\R^{n+1}}q(x,t)(-i)^{|\gamma|}\alpha^{|\gamma|}(x,t)^{\gamma}e^{-i\alpha (x,t).(\xi,\tau)}\,dx\,dt\Big|.
\end{array}$$
Hence one gets
$$ |\p^{\gamma}F_{\alpha}(\xi,\tau)|\leq \displaystyle\int_{\R^{n+1}}|q(x,t)|\alpha^{|\gamma|}(|x|^{2}+t^{2})^{\frac{|\gamma|}{2}}\,dx\,dt\leq \|q\|_{L^{1}(Q)}\alpha^{|\gamma|}(2T^{2})^{\frac{|\gamma|}{2}}\leq C\frac{|\gamma|!}{(T^{-1})^{|\gamma|}}e^{\alpha}.$$
Applying Lemma \ref{Lm3.6} on the set $O=E_{1}\cap B(0,1)$ with $M=C
e^{\alpha}$, $\eta=T^{-1}$, we may find a constant $\mu\in (0,1)$ such that
we have
$$|F_{\alpha}(\xi,\tau)|=|\widehat{q}(\alpha(\xi,\tau))|\leq Ce^{\alpha(1-\mu)}\|F_{\alpha}\|^{\mu}_{L^{\infty}(O)},\,\,\,\,\,\,(\xi,\tau)\in B(0,1).$$
Now the idea is to estimate the Fourier transform of $q$ in a
suitable ball. Bearing in mind that $\alpha E_{1}=E_{\alpha}$, we have for
all $(\xi,\tau)\in B(0,\alpha)$,
\begin{eqnarray}\label{3.25}
|\widehat{q}(\xi,\tau)|=|F_{\alpha}(\alpha^{-1}(\xi,\tau)|&\leq& Ce^{\alpha(1-\mu)}\|F_{\alpha}\|^{\mu}_{L^{\infty}(O)}\cr
&\leq& Ce^{\alpha(1-\mu)}\|\widehat{q}\|^{\mu}_{L^{\infty}(B(0,\alpha)\cap E_{\alpha})}\cr
&\leq& Ce^{\alpha (1-\mu)}\|\widehat{q}\|^{\mu}_{L^{\infty}(E_{\alpha})}.
\end{eqnarray}
The next step of the proof  is to get an estimate linking the 
coefficient $q$ to the measurement
$\Lambda_{A_{1},q_{1}}-\Lambda_{A_{2},q_{2}}$. To do that  we first
decompose the $H^{-1}(\R^{n+1})$ norm of $q$  as follows
$$\begin{array}{lll}
\|q\|_{H^{-1}(\R^{n+1})}^{\frac{2}{\mu}}\!\!&=&\!\!\Big(  \displaystyle\int_{|(\xi,\tau)|<\alpha}\!\!\!\!\!\!<(\xi,\tau)>^{-2}|\widehat{q}(\xi,\tau)|^{2}\,d\xi \,d\tau
+\displaystyle\int_{|(\xi,\tau)|\geq \alpha}\!\!\!\! \!\! <(\xi,\tau)>^{-2}|\widehat{q}(\xi,\tau)|^{2}\,d\tau\,d\xi\Big)^{\frac{1}{\mu}}\\
&\leq&C\para{\alpha^{n+1}\|\widehat{q}\|^{2}_{L^{\infty}(B(0,\alpha))}+\alpha^{-2}\|q\|_{L^{2}(\R^{n+1})}^{2}}^{\frac{1}{\mu}}.
\end{array}$$
It follows from (\ref{3.25}) and Lemma \ref{Lm3.5}, that
\begin{equation}
\|q\|_{H^{-1}(\R^{n+1})}^{\frac{2}{\mu}}\leq C\Big[\alpha^{\frac{n+1}{\mu}}e^{\frac{2\alpha (1-\mu)}{\mu}}\Big(e^{C\sigma}\eta^{2}+
e^{C\sigma}\|d\alpha_{A_{1}}-d\alpha_{A_{2}}\|^{2}_{L^{\infty}(\Omega)}+\frac{\alpha^{2}}{\sigma^{2}}+\frac{1}{\sigma^{2}}
\Big)+\frac{1}{\alpha^{\frac{2}{\mu}}}   \Big],
\end{equation}
where we have set $\eta=\|\Lambda_{A_{2},q_{2}}-\Lambda_{A_{1},q_{1}}\|$. In
light of  Theorem \ref{Thm1}, one gets
\begin{equation}\label{eq4.47}\|q\|_{H^{-1}(\R^{n+1})}^{\frac{2}{\mu}}\leq C\Big[
\alpha^{\frac{n+1}{\mu}}
e^{\frac{2\alpha(1-\mu)}{\mu}}\Big(e^{C\sigma}\eta^{2}+e^{c\sigma}\eta^{s}+e^{C\sigma}
|\log \eta|^{-2\mu
s}+\frac{\alpha^{2}}{\sigma^{2}}+\frac{1}{\sigma^{2}}\Big)+\frac{1}{\alpha^{\frac{2}{\mu}}}
\Big].
\end{equation}
 The above statements are valid provided $\sigma$ is sufficiently large. Then, we choose
$\alpha$ so large  that 
$\sigma=\alpha^{\frac{2\mu+n+3}{2\mu}}e^{\frac{\alpha(1-\mu)}{\mu}},$ and 
 hence $\alpha^{\frac{2\mu+n+1}{\mu}}e^{\frac{2\alpha
(1-\mu)}{\mu}}\sigma^{-2}=\alpha^{\frac{-2}{\mu}}$, so the  estimate
(\ref{eq4.47}) yields
\begin{equation}\label{eq4.48}\|q\|_{H^{-1}(\R^{n+1})}^{\frac{2}{\mu}}\leq C
\Big[ e^{Ce^{N\alpha}}(\eta^{2}+\eta^{s}+|\log \eta|^{-2\mu
s})+\alpha^{\frac{-2}{\mu}}\Big],
\end{equation} where $N$ depends on $\mu$ and $n$. Thus, if
$\eta\in (0,1)$, we have
\begin{equation}\label{eq4.49}\|q\|_{H^{-1}(\R^{n+1})}^{\frac{2}{\mu}}\leq
C\Big(e^{Ce^{N\alpha}}|\log \eta |^{-2\mu s} +\alpha^{\frac{-2}{\mu}}
\Big).
\end{equation}
 Finally, if $\eta$ is small enough, taking 
$\alpha=\frac{1}{N}\log \Big(\log|\log \eta|^{\frac{\mu s}{C}}\Big),$ 
we get from  (\ref{eq4.49}) that
$$\|q\|_{H^{-1}(\R^{n+1})}^{\frac{2}{\mu}}\leq
C\Big[ |\log\eta|^{-\mu s}+\Big[ \log\para{\log|\log\eta|^{\frac{\mu s}{c}}
}\Big]^{-\frac{2}{\mu}} \Big].$$
 This completes the proof of
Theorem\ref{Thm2}.

\appendix
\section{ Well-posedness of the magnetic Schr\"odinger equation}
In this section we will establish the existence, uniqueness and continuous
dependence with respect to the data, of the solution $u$ of the Schr\"odinger equation
(\ref{Eq1}) with non-homogeneous Dirichlet-boundary condition $f\in
H^{2,1}_{0}(\Sigma)$ and an initial data $u_{0}\in H_{0}^{1}(\Omega)\cap
H^{2}(\Omega)$.
\subsection{Proof of Theorem \ref{Thm1.1}}
  We decompose the
solution $u$ of the Schr\"odinger equation (\ref{Eq1}) as $u=u_{1}+u_{2}$,
with $u_{1}$ and $u_{2}$ are respectively solutions to
$$\left\{
  \begin{array}{ll}
   (i\p_{t}+\Delta_{A})u_{1}=0 , & \mbox{in}\,Q \\
    u_{1}(x,0)=0, & \mbox{in}\,\Omega \\
    u_{1}(x,t)=f, & \hbox{on}\, \Sigma
  \end{array}
\right.,\,\,\,\,\,\,\,\,\,\,\,\,\,\,\,\,\,\,\,\,\,\,\,\,\,\,\,\,\,\left\{
                                                                    \begin{array}{ll}
                                                                     (i\p_{t}+\Delta_{A}+q)u_{2}=-qu_{1} , & \mbox{in}\,Q \\
                                                                     u_{2}(x,0)=u_{0} , & \mbox{in}\,\Omega \\
                                                                     u_{2}=0 , & \mbox{on}\,\Sigma
                                                                    \end{array}
                                                                  \right.
$$
Using the fact that $f\in H^{2,1}_{0}(\Sigma)$, we can see from
\cite{[MB]}[Theorem 1.1] that
\begin{equation}\label{Aequation1}
u_{1}\in\mathcal{C}^{1}(0,T;H^{1}(\Omega)),
\end{equation}
 and 
\begin{equation}\label{Aequation2}
\|u_{1}\|_{\mathcal{C}^{1}(0,T;H^{1}(\Omega))}\leq C\|f\|_{H^{2,1}(\Sigma)}.
\end{equation}
 Moreover, we have $\p_{\nu}u_{1}\in
L^{2}(\Sigma)$, and we get  a constant $C>0$ such that
\begin{equation}\label{Aequation3}
\|\p_{\nu}u_{1}\|_{L^{2}(\Sigma)}\leq C\|f\|_{H^{2,1}(\Sigma)}.
\end{equation}
On the other hand, from \cite{MYE}[Lemma 2.1] , we conclude  the existence
of a unique solution
\begin{equation}\label{Aequation4}
u_{2}\in \mathcal{C}^{1}(0,T;L^{2}(\Omega))\cap \mathcal{C}(0,T;H^{2}(\Omega)\cap H^{1}_{0}(\Omega)),
\end{equation}
that satisfies
\begin{eqnarray}\label{Aequation5}
\|u_{2}(.,t)\|_{H^{1}_{0}(\Omega)}&\leq& C\para{\|qu_{1}\|_{W^{1,1}(0,T;L^{2}(\Omega)}+\|u_{0}\|_{H^{1}_{0}\cap H^{2}}}.\cr
&\leq& C\para{\|u_{0}\|_{H^{1}_{0}\cap H^{2}}+\|f\|_{H^{2,1}(\Sigma)}}.
\end{eqnarray}
Next,  we consider a $\mathcal{C}^{2}$ vector field $N$ satisfying
$$N(x)=\nu(x),\,\,\,\,x\in \Gamma,\,\,\,\,\,\,|N(x)|\leq 1,\,\,\,x\in \Omega.$$
Multiplying  the second Schr\"odinger equation  by $N.\nabla \overline{u}_{2}$
and integrating over $Q=\Omega\times (0,T)$ we get
$$\begin{array}{lll}
-\displaystyle\int_{0}^{T}\!\!\int_{\Omega}q\,u_{1} \,N.\nabla \overline{u}_{2}\,dx\,dt\!&=&\!i\displaystyle\int_{0}^{T}\!\!\int_{\Omega}\p_{t}u_{2}\,
 N.\nabla \overline{u}_{2} \,dx\,dt
+\displaystyle\int_{0}^{T}\int_{\Omega}\Delta u_{2} \,N.\nabla \overline{u}_{2}\,dx\,dt\\
&&+\displaystyle\int_{0}^{T}\int_{\Omega}(2iA.\nabla+idiv\,A-|A|^{2}+q)u_{2}\, N.\nabla \overline{u}_{2}\,dx\,dt=I_{1}+I_{2}+I_{3}.
\end{array}$$
By integrating with respect to
$t$ in the first term $I_{1}$, we get
$$\begin{array}{lll}
I_{1}&=&i\displaystyle\int_{\Omega}\Big[u_{2}(x,T)\,N.\nabla\overline{u}_{2}(x,T)-u_{2}(x,0) \,N.\nabla\overline{u}_{2}(x,0)\Big]\,dx\\
&&-i\displaystyle\int_{0}^{T}\int_{\Omega} N.\nabla(u_{2}\,\p_{t}\overline{u}_{2})\,dx\,dt+i\displaystyle\int_{0}^{T}
\int_{\Omega}\p_{t}\overline{u}_{2}\,N.\nabla u_{2}\,dx\,dt.
\end{array}$$
Therefore, bearing in mind that
$i\p_{t}\overline{u}_{2}=-q\,\overline{u}_{1}-\overline{\Delta_{A}u_{2}}-q\overline{u}_{2}$,
we get
$$\begin{array}{lll}
2\Re\, I_{1}&=& i\displaystyle\int_{\Omega}\Big[u_{2}(x,T)\, N.\nabla \overline{u}_{2}(x,T)-u_{0}\,N.\nabla \overline{u}_{0}\Big]\,dx-
\displaystyle\int_{0}^{T}\int_{\Omega} div \,N\,q\,u_{2}\overline{u}_{1}\,dx\,dt\\
&&-\displaystyle\int_{0}^{T}\int_{\Omega}div\,N\,q\,|u_{2}|^{2}\,dx\,dt+\displaystyle\int_{0}^{T}\int_{\Omega}\nabla_{A}(div\,N\,u_{2}).\nabla_{A}\overline{u}_{2}\,dx\,dt\\
&&-i\displaystyle\int_{0}^{T}\int_{\Gamma}  u_{2}\p_{t}\overline{u}_{2}\,d\sigma\,dt-\displaystyle\int_{0}^{T}\int_{\Sigma} \p_{\nu}\overline{u}_{2}(u_{2}\,div\,N)\,d\sigma\,dt.
\end{array}$$
As the last term vanishes since $u_{2}=0$ on $\Sigma$, we deduce from  (\ref{Aequation5}) that
$$|\Re\, I_{1}|\leq C
\para{\|f\|_{H^{2,1}(\Sigma)}^{2}+\|u_{0}\|^{2}_{H^{1}_{0}\cap H^{2}}}.$$
On the other hand, by Green's Formula,  we have
$$\begin{array}{lll}
I_{2}&=&-\displaystyle\int_{0}^{T}\int_{\Omega} \nabla u_{2}\nabla (N.\nabla\overline{u}_{2})\,dx\,dt+\displaystyle\int_{0}^{T}\int_{\Gamma}\p_{\nu}u_{2}(N.\nabla \overline{u}_{2})\,d\sigma\,dt\\
&=&-\displaystyle\int_{0}^{T}\int_{\Omega}\nabla u_{2}.\nabla (N.\nabla\overline{u}_{2})\,dx\,dt+\displaystyle\int_{0}^{T}\int_{\Gamma}|\p_{\nu}u_{2}|^{2}\,d\sigma\,dt,
\end{array}$$
So, we get
$$\begin{array}{lll}
I_{2}&=&\displaystyle\int_{0}^{T}\int_{\Gamma}|\p_{\nu}u_{2}|^{2}\,d\sigma\,dt
-\displaystyle\frac{1}{2}\displaystyle\int_{0}^{T}\int_{\Omega} div(|\nabla u_{2}|^{2} N)\,dx\,dt\\
&&+\displaystyle\frac{1}{2}\displaystyle\int_{0}^{T}\int_{\Omega}|\nabla u_{2}|^{2}\,div\,N\,dx\,dt
-\displaystyle\int_{0}^{T}\displaystyle\int_{\Omega}DN(\nabla u_{2},\nabla \overline{u_{2}})\,dx\,dt.
\end{array}$$
Thus, we have
$$\begin{array}{lll}
I_{2}&=&\displaystyle\int_{0}^{T}\int_{\Gamma}|\p_{\nu}u_{2}|^{2}
-\displaystyle\frac{1}{2}\displaystyle\int_{0}^{T}\int_{\Gamma}|\nabla u_{2}|^{2}\,N.\nu\,d\sigma\,dt\\
&&+\displaystyle\frac{1}{2}\displaystyle\int_{0}^{T}\int_{\Omega}|\nabla u_{2}|^{2}\,div\,N\,dx\,dt
-\displaystyle\int_{0}^{T}\int_{\Omega}DN(\nabla u_{2},\nabla u_{2})\,dx\,dt.
\end{array}$$
Next, using the fact that
$$|\nabla u_{2}|^{2}=|\p_{\nu}u_{2}|^{2}+|\nabla_{\tau}u_{2}|^{2}=|\p_{\nu}u_{2}|^{2},\,\,\,\,\,\,\,\,\,x\in\Gamma,$$
where $\nabla_{\tau}$ is the tangential gradient on $\Gamma$, we obtain
$$\begin{array}{lll}
\Re\, I_{2}&=&\displaystyle\frac{1}{2}\displaystyle\int_{0}^{T}\int_{\Omega}|\p_{\nu}u_{2}|^{2}d\sigma\,dt+\frac{1}{2}\displaystyle\int_{0}^{T}|\nabla u_{2}|^{2}\,div\,N\,dx\,dt\\
&&-\displaystyle\int_{0}^{T}\int_{\Omega}DN(\nabla u_{2},\nabla \overline{u}_{2})\,dx\,dt.
\end{array}$$
Moreover, by (\ref{Aequation5}), it is easy to see that
$$|\Re \, I_{3}|\leq C\para{\|f\|_{H^{2,1}(\Sigma)}^{2}+\|u_{0}\|^{2}_{H^{1}_{0}\cap H^{2}}},$$
so that, we deduce from the above statements that
$$\begin{array}{lll}
\|\p_{\nu}u_{2}\|_{L^{2}(\Sigma)}&\leq& C\para{\|f\|_{H^{2,1}(\Sigma)}+\|u_{0}\|_{H^{1}_{0}\cap H^{2}}}.
\end{array}$$
From the above reasoning, we conclude that $u=u_{1}+u_{2}\in
\mathcal{C}(0,T;H^{1}(\Omega))$, $\p_{\nu}u\in L^{2}(\Sigma)$ and we have
$$\|u(.,t)\|_{H^{1}(\Omega)}+\|\p_{\nu}u\|_{L^{2}(\Sigma)}\leq C\para{\|f\|_{H^{2,1}(\Sigma)}+\|u_{0}\|_{H^{1}_{0}\cap H^{2}}}.$$
\section{ Some fundamental statements} In this section, we 
collect several technical results that are needed in the proof of
the main results. We first introduce the following notations. Let $P(D)$ be a
differential operator with $D=-i(\p_{t},\p_{x})$. We denote by
$$ \widetilde{P}(\xi,\tau)=\Big( \sum_{k\in \mathbb{N}}\sum_{\alpha\in\mathbb{N}^{n}}|\p^{k}_{\tau}\p^{\alpha}_{\xi}
P(\xi,\tau)|^{2} \Big)^{\frac{1}{2}},\,\,\,\,\,\,\,\,\,\,\xi\in
\R^{n},\,\tau\in\R.$$ For $1\leq p\leq \infty$, we define the space
$$B_{p,\widetilde{P}}=\{f\in S'(\R^{n+1}),\,\,\,\widetilde{P} \mathcal{F}(f)\in L^{p}(\R^{n+1})\}, $$
equipped with the following norm
$$\|f\|_{B_{p,\widetilde{P}}}=\|\widetilde{P}\mathcal{F}(f)\|_{L^{p}(\R^{n+1})}.$$
We finally  denote by
$$B_{p,\widetilde{P}}^{loc}=\{f\in S'(\R^{n+1}), \,\varphi f\in B_{p,\widetilde{P}},\,\,\forall\,\varphi\in\mathcal{C}_{0}^{\infty}(\R^{n+1})\}.$$
 We start by recalling some known results of H\"ormander:
\begin{Lemm}\label{Blm1}
Let $u\in B_{\infty,\widetilde{P}}$ and $v\in
\mathcal{C}^{\infty}_{0}(\R^{n+1})$. Then, we have  $uv\in B_{\infty,\widetilde{P}}$, 
and 
$$\|uv\|_{\infty,\widetilde{P}}\leq C \|u\|_{\infty,\widetilde{P}},$$
where the positive constant $C$  depends only on $v,\,n$ and the degree of
$P$.
\end{Lemm}
\begin{Lemm}\label{Blm2} Any differential operator $P(D)$ admits a fundamental solution
$F\in B_{\infty,\widetilde{P}}^{loc}$ \,satisfying\,\,\,\,
\,$\frac{F}{\cosh|(x,t)|}\in B_{\infty,\widetilde{P}}$. Moreover, it verifies
$$
\|\frac{F}{\cosh|(x,t)|}\|_{\infty,\widetilde{P}}\leq C,
$$
where $C$ is a positive constant that depends only on $n$ and the degree of
$P$.
\end{Lemm}
Our first goal in this section is to prove the following theorem:
\begin{theorem}
 Let $P\neq0$ be a differential operator. Then for all $k\in \mathbb{N}$, there exists a linear operator
$$E: L^{2}(0,T;H^{k}(\Omega))\rightarrow L^{2}(0,T;H^{k}(\Omega)),$$
 such that:
 \begin{enumerate}
   \item $P(D)Ef=f,$ for any $f\in L^{2}(0,T;H^{k}(\Omega)).$
   \item For any linear differential operator with constant coefficient
       $Q(D)$ such that
       $\displaystyle\frac{|Q(\xi,\tau)|}{\widetilde{P}(\xi,\tau)}$ is
       bounded, we have $Q(D) E\in B (L^{2}(0,T;H^{k}(\Omega)))$ and
       $$\|Q(D) E f\|_{L^{2}(0,T;H^{k}(\Omega))}\leq C\sup_{\R^{n+1}}\frac{|Q(\xi,\tau)|}{\widetilde{P}(\xi,\tau)}\,\|f\|_{L^{2}(0,T;H^{k}(\Omega))}, $$
 \end{enumerate}
where $C$ depends only on the degree of $P$, $\Omega$ and $T$.
\end{theorem}
\begin{proof}{}
Let $f\in L^{2}(0,T;H^{k}(\Omega))$. There exists an
extension operator
$$\begin{array}{ccc}
S:L^{2}(0,T;H^{k}(\Omega))&\longrightarrow &L^{2}(0,T;H^{k}(\R^{n}))\\
f&\longmapsto& \widetilde{f},
\end{array}$$
such that for all $t\in (0,T)$, we have
$\widetilde{f}(.,t)_{|\Omega}=f(.,t)$. Next, we introduce 
$$\widetilde{f}_{0}=\left\{
                           \begin{array}{ll}
                             \widetilde{f}, & t\in(0,T),\,\,\,\,x\in\R^{n}\\
                             \\
                             0, & t\notin (0,T),\,\,\,\,x\in\R^{n}.
                           \end{array}
                         \right.$$
So, we have $\widetilde{f}_{0|Q}=f$. Let $R>0$  and $V$ be a neighborhood
of $\overline{Q}$. We consider $\psi\in \mathcal{C}_{0}^{\infty}(\R^{n+1})$
such that $\psi_{|V}=1$ and satisfying supp $\psi\subset
B(0,R)\subset\R^{n+1}$. Let $F$ be a fundamental solution of $P$. We consider
the following operator
$$\begin{array}{rrr}
E:L^{2}(0,T;H^{k}(\Omega))&\longrightarrow& L^{2}(0,T;H^{k}(\Omega))\\
f&\longmapsto& E(f)=(F\ast\psi \widetilde{f}_{0})_{|Q}
\end{array}$$
Since $P(D)(F\ast\psi\widetilde{f}_{0})=\psi\widetilde{f}_{0}$, then we
 clearly have
$$P(D)Ef= (\psi\widetilde{f}_{0})_{|Q}=f.$$
We turn now to proving the second point. For this purpose, we consider
$\varphi\in \mathcal{C}_{0}^{\infty}(\R^{n+1})$ such that $\varphi=1$ on a
neighborhood of the closure of 
 $\{x-y,\,\, \,x,\,y\in Q\}.$
We  can easily  verify that
$$(F\ast \psi\widetilde{f}_{0})_{|Q}=(\varphi F\ast\psi\widetilde{f}_{0})_{|Q}.$$
 The last
identity entails that for all $\alpha\in \mathbb{N}^{n}, $ such that
$|\alpha|\leq k$, we have
\begin{eqnarray}\label{Beq}
\|\p^{\alpha} Q(D)Ef\|_{L^{2}(Q)}&=&\|Q(D) \p^{\alpha} (F\ast \psi\widetilde{f}_{0})\|_{L^{2}(Q)}\cr
&=&\|Q(D) \varphi F\ast \p^{\alpha}(\psi\widetilde{f}_{0})\|_{L^{2}(Q)}\cr
&\leq&\|Q(D) \varphi F\ast \p^{\alpha}(\psi\widetilde{f}_{0})\|_{L^{2}(\R^{n+1})}\cr
&\leq& \| \mathcal{F}\big(Q(D) \varphi F\ast \p^{\alpha}(\psi\widetilde{f}_{0})\big)\|_{L^{2}(\R^{n+1})}\cr
&\leq& \|Q(\xi,\tau)\mathcal{F}(\varphi F) \mathcal{F}(\p^{\alpha}(\psi \tilde{f}_{0}))\|_{L^{2}(\R^{n+1})}\cr
&\leq&\|Q(\xi,\tau)\mathcal{F}(\varphi F)\|_{L^{\infty}(\R^{n+1})}\|\p^{\alpha}(\psi \tilde{f}_{0})\|_{L^{2}(\R^{n+1})}\cr
&\leq&\|Q(\xi,\tau)\mathcal{F}(\varphi F)\|_{L^{\infty}(\R^{n+1})}\|\p^{\alpha} f\|_{L^{2}(Q)}.
\end{eqnarray}
Using the fact that
$$ Q(\xi,\tau)\mathcal{F}(\varphi
F)=\frac{Q(\xi,\tau)}{\widetilde{P}(\xi,\tau)}\widetilde{P}(\xi,\tau)
\mathcal{F}\Big(\varphi \cosh|(x,t)|\frac{F}{\cosh|(x,t)|}\Big),$$
 we deduce from Lemma \ref{Blm1} and Lemma \ref{Blm2} that
\begin{equation}\label{Beq (0.2)}
\|Q(\xi,\tau)\mathcal{F}(\varphi F)\|_{L^{\infty}(\R^{n+1})}\leq C \sup_{(\xi,\tau)\in\R^{n+1}}\frac{|Q(\xi,\tau)|}{\widetilde{P}(\xi,\tau)}.
\end{equation}
Then from (\ref{Beq}) and (\ref{Beq (0.2)}), we get
\begin{equation}\label{Beq3}
\|\p^{\alpha} Q(D)Ef\|_{L^{2}(Q)}\leq C \sup_{(\xi,\tau)\in\R^{n+1}}\frac{|Q(\xi,\tau)|}{\widetilde{P}(\xi,\tau)}\|f\|_{L^{2}(0,T;H^{k}(\Omega))},\,\,\,\,\,
 \forall\,\alpha\in \mathbb{N}^{n},\,\, |\alpha|\leq k.
\end{equation}
Thus, we find that
$$\|Q(D) E f\|_{L^{2}(0,T;H^{k}(\Omega))}\leq C\sup_{\R^{n+1}}\frac{|Q(\xi,\tau)|}{\widetilde{P}(\xi,\tau)}\,\|f\|_{L^{2}(0,T;H^{k}(\Omega))}, $$
which completes the proof of the lemma.
\end{proof}
Finally, we establish the following statement:
 \begin{Lemm}\label{lm6.4}
 Let $\Omega\subset \R^{n}$ be a simply connected domain, and let $A\in \mathcal{C}^{2}(\Omega,\R^{n})$ be such that $A_{| \Gamma}=0$. Then, for $p>n$,
there exists  a function $\varphi\in \mathcal{C}^{3}(\Omega)$ such that
$\varphi_{|\Gamma}=0$ and $A'\in W^{1,p}(\Omega, \R^{n})$, satisfying\,\,
 $A=A'+\nabla \varphi, \,\,\,\,A'\wedge \nu=0,$ and $\mbox{div} A'=0.$
 Moreover, there exists a constant $C>0$, such that
 \begin{equation}\label{Equation C3.37}
\|A'\|_{W^{1,p}(\Omega)}\leq C\, \|\mbox{curl} \,A'\|_{L^{p}(\Omega)}.
\end{equation}
 \end{Lemm}
 \begin{proof}{}
 Let $\varphi$ be the solution of the following problem
 \begin{equation}
\left\{
     \begin{array}{ll}
       \Delta \varphi=\mbox{div} A, & \mbox{in}\,\Omega \\
       \varphi=0, & \mbox{in}\,\Gamma.
     \end{array}
   \right.
\end{equation}
Then, setting  $A'=A-\nabla \varphi$,  using the fact that $A_{|
\Gamma}=\varphi_{|\Gamma}=0,$ one gets
$$A'\wedge \nu=A\wedge \nu-\nabla \varphi\wedge\nu=0,\,\,\,\,\,\mbox{and}\,\,\,\,\,\,\,\mbox{div}A'=0. $$
In order to prove (\ref{Equation C3.37}), we argue by
contradiction. We assume that for all $k\geq 1$  there exists a non-null
$\widetilde{A'}_{k} \in W^{1,p}(\Omega)$ such that
\begin{equation}\label{Equation C3.39}
\|\widetilde{A'_{k}}\|_{W^{1,p}(\Omega)}\geq k\,\|\mbox{curl}\,\widetilde{A'}_{k}\|_{L^{p}(\Omega)}.
\end{equation}
We set
$A'_{k}=\displaystyle\frac{\widetilde{A'}_{k}}{\|\widetilde{A'}_{k}\|_{W^{1,p}(\Omega)}}.$
Then we have $ \|A'_{k}\|_{W^{1,p}(\Omega)}=1$ and $k
\,\|\mbox{curl}\,A'_{k}\|_{L^{p}(\Omega)}\leq 1.$
 In view of the weak compactness theorem, there exists a subsequence of $(A'_{k})_{k}$ such that
$A'_{k}\rightharpoonup A'\,\,\,\, \mbox{in} \,\,W^{1,p}(\Omega).$ Using
the fact that $W^{1,p}(\Omega)\hookrightarrow L^{p}(\Omega),$ we deduce that
$A'_{k}\rightarrow A' \,\,\,\,\,\mbox{in} \,\, L^{p}(\Omega).$ As a
consequence, we have
$$
\|A'\|_{W^{1,p}(\Omega)}=1\,\,\,\,\,\,\mbox{and}\,\,\,\,\|\mbox{curl}\,A'\|_{L^{p}(\Omega)}=0.
$$
This  entails that  there exists
$\eta\in W^{1,p}(\Omega)$ such that $A'=\nabla \eta$. Then, using  the fact that div $A'=0$
and $A'\wedge \nu=0$, we deduce that there exists a constant $\lambda\in \R$
such that
$$\left\{
  \begin{array}{ll}
    \Delta \eta=0, & \mbox{in}\,\, \Omega \\
    \eta=\lambda, & \mbox{in} \,\,\Gamma.
  \end{array}
\right.
$$
Finally, using the fact that $\Omega$ is a simply connected domain we
conclude that $\eta=\lambda$ in $\overline{\Omega}$.  This entails that
$A'=0$ and contradicts the fact that $\|A'\|_{W^{1,p}(\Omega)}=1$.
\end{proof}
As a consequence of Lemma \ref{lm6.4}, we have the following result
\begin{Lemm}\label{rot}
Let $\Omega\subset\R^{n}$ be a simply connected domain, and let $A\in
\mathcal{C}^{2}(\Omega,\R^{n})$ such that $A_{| \Gamma}=0$. If we further
assume that div $A=0$, then the following estimate
  $$
\|A\|_{W^{1,p}(\Omega)}\leq C\, \|\mbox{curl} \,A\|_{L^{p}(\Omega)},
$$
holds true for some positive constant $C$ which is independent of $A$. 
\end{Lemm}%


\textbf{Acknowledgements}
The author would like to thank Pr. Mourad Bellassoued and Pr. Eric Soccorsi for many helpful suggestions they made and for their careful reading of the manuscript.

\end{document}